\documentclass[reqno,11pt]{amsart}
\usepackage{amsmath,amssymb,amsthm, verbatim}
\usepackage{url}
\usepackage[usenames, dvipsnames]{color}

\usepackage[letterpaper,hmargin=1in,vmargin=1in]{geometry}

\usepackage{graphicx}
\usepackage{enumerate,pinlabel}
\usepackage{mathrsfs,graphicx,url}
\usepackage[usenames,dvipsnames]{xcolor}
\usepackage[colorlinks=true,linkcolor=Black,citecolor=Black, urlcolor=Black, pagebackref]{hyperref}

\numberwithin{equation}{section}

\theoremstyle{plain}
\newtheorem{theorem}{Theorem}[section]

\newtheorem*{theorem*}{Theorem}
\newtheorem*{lemma*}{Lemma}
\newtheorem{lemma}[theorem]{Lemma}
\newtheorem{proposition}[theorem]{Proposition}
\newtheorem{corollary}[theorem]{Corollary}

\theoremstyle{definition}

\theoremstyle{remark}
\newtheorem{remark}[theorem]{Remark}

\newcommand{\ep}{\varepsilon}

\newcommand{\R}{\mathbb{R}}

\newcommand{\C}{\mathbb{C}}

\newcommand\supp{\mathop{\rm supp}}

\newcommand\sgn{\mathop{\rm sgn}} 
\newcommand\real{\mathop{\rm Re}}
\newcommand\imag{\mathop{\rm Im}}
\newcommand*{\defeq}{\mathrel{\vcenter{\baselineskip0.5ex \lineskiplimit0pt

                     \hbox{\scriptsize.}\hbox{\scriptsize.}}}%
                     =}
                     
\newcommand*{\qefed}{=\mathrel{\vcenter{\baselineskip0.5ex \lineskiplimit0pt

                     \hbox{\scriptsize.}\hbox{\scriptsize.}}}%
                     }

\definecolor{jeffColor}{RGB}{102, 0, 204}

\title[Magnetic potentials on the line]{Semiclassical estimates for the magnetic Schr\"odinger operator on the line}

\begin{document}

\author{Andr\'es Larra\'in-Hubach}
\address{Department of Mathematics, University of Dayton, Dayton, OH 45469-2316, USA}
\email{alarrainhubach1@udayton.edu}
\author{Jacob Shapiro}
\address{Department of Mathematics, University of Dayton, Dayton, OH 45469-2316, USA}
\email{jshapiro1@udayton.edu}


\begin{abstract}
We prove a weighted Carleman estimate for a class of one-dimensional, self-adjoint Schr\"odinger operators $P(h)$ with low regularity electric and magnetic potentials, where $h > 0$ is a semiclassical parameter. The long range part of either potential has bounded variation. The short range part of the magnetic potential belongs to $L^1(\R) \cap L^2(\R)$, while the short range part of the electric potential is a finite signed measure.  The proof is a one dimensional instance of the energy method, which is used to prove Carleman estimates in higher dimensions and in more complicated geometries. The novelty of our result lies in the weak regularity assumptions on the coefficients. As a consequence of the Carleman estimate, we establish optimal limiting absorption resolvent estimates for $P(h)$. We also present standard applications to the distribution of resonances for $P(1)$ and to associated evolution equations.
\end{abstract}
\maketitle 
\author
\section{Introduction and statement of results} \label{intro section}
The goal of this paper is to prove sharp limiting absorption resolvent estimates for one dimensional magnetic Schr\"odinger operators with low regularity coefficients. More precisely, we consider operators of the form
\begin{equation} \label{mag schro oned}
\begin{split}
P(h) &= \beta(x,h)( -h^2\partial_x(\alpha(x, h) \partial_x) + hb(x,h)D_x + hD_xb(x,h)) \\
&+V(x,h) : L^2(\R ; \beta^{-1} dx) \to L^2(\R; \beta^{-1}dx ), \qquad h > 0,
\end{split}
\end{equation}
with real valued coefficients, where $D_x \defeq -i \partial_x$ and $h$ is a semiclassical parameter. 

First, we fix the assumptions on the coefficients in \eqref{mag schro oned}. The electric potential $V$ and magnetic potential $b$ may depend on $h$ and they decompose
\begin{gather} 
V(x,h) = V_0(x,h) + V_1(x,h), \label{decompose electric potential} \\
b(x,h) = b_0(x,h) + b_1(x,h), \label{decompose magnetic potential}
\end{gather}
where
\begin{gather} 
\text{$V_0$ is a finite signed Borel measure on $\R$,} \label{V0 Borel measure} \\
b_0 \in L^1(\R)\cap L^2(\R),
\end{gather}
and
\begin{equation} \label{V1 BV}
\text{$V_1, \, b_1$ have bounded variation (BV),} \\
\end{equation}
in the sense that each is a difference of bounded nondecreasing functions. As for the coefficients $\alpha$ and $\beta$, which may also depend on $h$, we suppose 

 \begin{equation} \label{alpha beta BV}
\text{$\alpha, \, \beta : \R \to (0,\infty)$ have BV,}
\end{equation}
and for all $h > 0$,
\begin{equation} \label{inf alpha beta}
\inf \alpha(\cdot, h), \, \inf \beta(\cdot, h)>0.
\end{equation}
Recall that the distributional derivative of a BV function $f : \R \to \R$ is a finite signed Borel measure, which we denote by $df$. We review this and other standard facts about BV functions in Section \ref{bv review section}. 

Prior to studying resolvent estimates, it is important to describe a domain on which \eqref{mag schro oned} is self-adjoint. Self-adjointness for large classes of Sturm-Liouville operators with singular coefficients was addressed systematically in \cite{hrmy01, hrmy12, ecte13,lsw24}. In \cite{aley12}, self-adjointness of \eqref{mag schro oned} was shown in the case $\alpha = \beta = 1$, $V_1 = b_0 = 0$, and $V_0, b_1,  db_1 \in L^1(\R)$. We approach this task from an elementary viewpoint, using the calculus of BV functions and $L^2$-based estimates to characterize the realization of \eqref{mag schro oned} stemming from the sesquilinear form
\begin{equation*}
q(u,v) \defeq  h^2 \int_{\R} \alpha \overline{u}' v' dx + ih \int_{\R} b(v \overline{u}' - \overline{u} v') dx + \int_{\R} \overline{u}_c v_c \beta^{-1} V_0, 
\end{equation*}
defined on the Sobolev space $H^1(\R)$.  In Appendix \ref{self adj oned section}, we show \eqref{mag schro oned} is self-adjoint with respect to 
  \begin{equation} \label{D}
\begin{gathered}
\mathcal{D} \defeq \{ u \in H^1(\R) :  h \alpha u' + ibu \text{ belongs to $L^\infty(\R)$, has locally bounded variation, and} \\
\beta d( -h^2 \alpha u' - i h bu)  + h \beta b D_xu +  u_cV \in L^2(\R) \},
\end{gathered}
\end{equation} 
where $u_c$ denotes the unique continuous representative of $u$. For each $u \in \mathcal{D}$, $P(h)u \in L^2(\beta^{-1}dx)$ is then defined distributionally by
\begin{equation} \label{P(h)u}
P(h)u \defeq \beta d( -h^2 \alpha u' - i h bu)  + h \beta b D_xu +  u_cV.
\end{equation}

Our principal result is a Carleman estimate for \eqref{mag schro oned}.  

\begin{theorem}\label{Carleman mag schro 1D thm}
Fix  $s > 1/2$ and $E_{\max}, \, h_0, \, \ep_0 > 0$. Let the semiclassical parameter $h$ vary in $(0,h_0]$. Suppose $\{E(h)\}_{h \in (0,h_0]} \subseteq [-E_{\max}, E_{\max}]$ and $\{\ep(h)\}_{h \in (0, h_0]} \subseteq [-\ep_0, \ep_0]$ are families that may depend on $h$. Let $P(h)$ be given by \eqref{mag schro oned}, with coefficients satisfying \eqref{decompose electric potential} through \eqref{inf alpha beta}, and equipped with domain $\mathcal{D}$ as in \eqref{D}. 

Suppose further that for each $h \in (0, h_0]$, there exists $R_1(h) \ge 0$ so that 
\begin{equation} \label{general inf}
\inf_{|x| \ge R_1(h)}(\beta^{-1}( x, h) \alpha(x, h)(E(h) -V_1(x, h))  + b^2_1(x, h)) > 0.
\end{equation}
Fix a phase  $\varphi = \varphi( \cdot, h)\in C^0(\R ; [0, \infty))$ such that $\varphi(0) = 0$, $\varphi$ is even, $\varphi \in C^\infty(0,\infty)$, and 
\begin{equation} \label{varphi}
\partial_x \varphi(x) = \begin{cases} 
 k(h), &  0 < x \le R_1(h), \\
 0 & x >  2R_1(h),
\end{cases}
\end{equation}
where $k(h) \ge 0$ is chosen large enough so that
\begin{equation} \label{tau}
\tau = \tau(h) \defeq \inf_{x \in \R} ((\beta^{-1}(x, h)  \alpha(x, h)(E(h) -V_1(x, h)) + \alpha^2(x, h)(\partial_x \varphi)^2(x, h) + b^2_1(x, h)) > 0.
\end{equation}

For each $h \in (0, h_0]$, there is $C(h) > 0$ depending on $s, \, E_{\max}, \, h_0, \, \ep_0, \,h$, and the coefficients of \eqref{mag schro oned}, so that for all $v \in \mathcal{D}$ with $\langle x \rangle^{2s} (P(h) - E(h) - i\ep(h)) v \in L^2(\R)$,
\begin{equation} \label{weighted est mag schro}
\begin{split}
  \int_{\R} \langle x \rangle^{-2s}(|e^{\varphi(x)/h}v|^2 &+ |(h\alpha \partial_x + ib)e^{\varphi(x)/h}v|^2) dx \\
    &\le  C(h) \int_{\R} \langle x \rangle^{2s} |(P(h) - E(h) - i\ep(h)) v |^2 dx +   |\ep(h)| C(h)  \int_{\R} | v|^2 dx,
  \end{split}
\end{equation}
where $\langle x \rangle \defeq (1 + |x|^2)^{1/2}$.
\end{theorem}

 \begin{remark}
 In one dimension, a change of variable can transform a first order perturbation into a zeroth order perturbation, reducing the proof of estimates like \eqref{weighted est mag schro} to the case $P(h) = -h^2 \partial^2_x + V$. But the most general transformation the authors know of  \cite[Proposition 1.1]{d'anfa06} requires the higher order coefficients to be absolutely continuous and have short range derivatives.
 \end{remark}

\begin{remark}
 If $V_1(\cdot, h)$ exhibits ``long range" decay, i.e., $\limsup_{|x| \to \infty} V_1(x, h) = 0$, then \eqref{general inf} is satisfied for any $E(h) > 0$. Hence \eqref{weighted est mag schro} with $\ep(h) = 0$ implies $P(h)$ has no positive eigenvalues. We can also rule out eigenvalues $E(h) < 0$ for which \eqref{general inf} holds. Absence of positive eigenvalues for Schr\"odinger operators with locally $H^{-1}$ potentials that have $L^1$ decay was proved in \cite{lsw24}.
\end{remark}

If the coefficients depend on $h$ in a uniform way, one can better track constants in the estimates leading to \eqref{weighted est mag schro}. Thus $C(h)$ may be specified more precisely, yielding sharp exterior and exponential limiting absorption resolvent estimates. 

In the sequel, $|\mu|$ stands for the total variation of a finite signed Borel measure $\mu = \mu(h)$ which may depend on $h$. We put $\| \mu \| \defeq |\mu|(\R)$.  Recall that $|\mu|$ is defined as the sum of the positive and negative variations of $\mu$, as supplied by the Jordan decomposition theorem \cite[Theorem 3.4]{fo}. 

\begin{corollary} \label{ext resolv est corollary}
Assume the hypotheses of Theorem \ref{Carleman mag schro 1D thm} and the following: 
\begin{enumerate}
\item $\|V_1(\cdot, h)\|_{L^\infty}$, $\|dV_1(h)\|$, $\|b_1(\cdot, h)\|_{L^\infty}$, $\|db_1(h)\|$, $\|d\alpha(h) \|$, $\|d\beta(h)\|$, $\sup \alpha(\cdot, h)$,\\ and $\sup \beta(\cdot, h)$ are bounded uniformly for $h \in (0,h_0]$.

\item There is $c > 0$ independent of $h$ so that for all $h \in (0,h_0]$,
\begin{equation} \label{unif inf alpha beta}
\inf \alpha(\cdot, h), \, \inf \beta(\cdot, h) > c. 
\end{equation}
\item There is $R_1 > 0$ independent of $h$ so that
\begin{equation} \label{limsup 1D}
\tau_1 \defeq \inf_{h \in (0,h_0], \, |x| \ge R_1} (\beta^{-1}(x,h)\alpha(x, h)(E(h) - V_1(x , h)) + b_1^2(x, h)) > 0.
\end{equation}
\item There is $R_0 > 0$ independent of $h$ so that $V_0(x; h)$ and $b_0(x;h)$ are supported in $[-R_0, R_0]$ for all $h \in (0,h_0]$.
\end{enumerate}
Then there exists $C_0> 0$ independent of $h$ so that for all $h \in (0,h_0]$ and $\ep(h) \neq 0$,
\begin{equation} \label{ext resolv est mag schro}
 \| \langle x \rangle^{-s}\mathbf{1}_{> R} (P(h) - E(h) - i\varepsilon(h))^{-1} \mathbf{1}_{> R} \langle x \rangle^{-s}\|_{L^2(\mathbb R) \to L^2(\mathbb R)} \le \frac{C_0}{h},
\end{equation}
where $\mathbf{1}_{> R}$ is the indicator function of $\{ x : |x| > R \}$, and $R = \max(R_0, R_1)$. 
\end{corollary}

\begin{remark}
Note that \eqref{unif inf alpha beta} is a stronger condition than \eqref{inf alpha beta} because the lower bound $c$ is required to hold uniformly for $h \in (0,h_0]$. A familiar setting in which \eqref{limsup 1D} holds is when $E > 0, \alpha, \beta, V_1$ and $b_1$ are all independent of $h$, and $\limsup_{|x| \to \infty} V_1(x) = 0$. 
\end{remark}

\begin{corollary} \label{exp resolv est corollary}
Assume the hypotheses of Theorem \ref{Carleman mag schro 1D thm}, the first three conditions of Corollary \ref{ext resolv est corollary}, and 
\begin{enumerate}
\setcounter{enumi}{4}
\item $\|V_0(h)\|$, $\|b_0(\cdot, h)\|_{L^1}$, and  $\|b_0(\cdot, h)\|_{L^2}$ are bounded uniformly for $h \in (0,h_0]$.
\end{enumerate}
Then there is $C_1 > 0$ independent of $h$ so that for all $h \in (0, h_0]$ and $\ep(h) \neq 0$,
\begin{equation} \label{resolv est mag schro}
 \|\langle x \rangle^{-s}  (P(h) - E(h) - i\ep(h))^{-1}\langle x \rangle^{-s}\|_{L^2(\mathbb R) \to L^2(\mathbb R)} \le e^{C_1/h}.
\end{equation}
\end{corollary}

\begin{remark}
A feature of Corollaries \ref{ext resolv est corollary} and \ref{exp resolv est corollary}, which is a little bit stronger than typical semiclassical estimates, is that given \textit{any} $h_0 > 0$, rather than just \textit{some} $h_0 > 0$ small enough, the bounds \eqref{ext resolv est mag schro} and \eqref{resolv est mag schro} hold uniformly for all $h \in (0,h_0]$. In general the constants $C_0, \, C_1$ will depend on $h_0$ and grow as $h_0 \to \infty$.
\end{remark}

Estimates like \eqref{ext resolv est mag schro} and \eqref{resolv est mag schro} were proved previously for $\alpha,\, \beta = 1$, $V_1, \, b_0, \, b_1 = 0$ \cite{lash24}. Thus, the novelty of our present work is that \eqref{weighted est mag schro} implies optimal semiclassical resolvent bounds for a wide class of one dimensional operators that can have singular short range coefficients and discontinuous long range coefficients. 

The program of establishing optimal semiclassical estimates was initiated by Burq \cite{bu98, bu02}, where he proved an exponential bound like \eqref{resolv est mag schro} for a large class of operators with smooth coefficients in all dimensions. A weaker version of the exterior bound \eqref{ext resolv est mag schro} was given in \cite{bu02}, and then refined by Cardoso and Vodev \cite{cavo02}. Subsequent works have reduced the regularity and decay needed to have optimal bounds, see e.g. \cite{vo13, da14, ddh16, sh19, vo20c, gash22b, ob24}. Of particular salience to the present work is Vodev's recent paper \cite{vo25}, which establishes optimal semiclassical estimates in dimension three and higher for Schr\"odinger operators with long range electric and magnetic potentials.
 
To obtain \eqref{ext resolv est mag schro} and \eqref{resolv est mag schro} in dimension greater than one, the proofs usually require the long range coefficients to be Lipschitz continuous in the radial variable. For short range, $L^\infty$ electric or magnetic potentials, only versions of \eqref{resolv est mag schro} with additional losses are known \cite{klvo19, vo19, vo20a, vo20b, sh20, gash22a, sh24, vo25}.

If each of $V_0$, $V_1$ and $b_0$ are supported in $[-R_0, R_0]$, then from \eqref{limsup 1D} we may take $R_1 = R_0$, thus $R = R_0$ in \eqref{ext resolv est mag schro}. In higher dimensions, Datchev and Jin \cite{daji20} gave examples of smooth, compactly supported potentials $V$ for which the exterior estimate holds only if the chosen weight vanishes on a set much larger than the support of $V$.

The proof of Theorem \ref{Carleman mag schro 1D thm} in Section \ref{pf of main thm section} is structured as a positive commutator-type argument in the context of the so-called energy method. This strategy has long been used to prove Carleman and related estimates \cite{cavo02, da14, klvo19, dash20, gash22b, ob24}. As we work in one dimension, we begin from a pointwise-defined energy
\begin{equation} \label{F}
\begin{gathered} 
F(x) \defeq |(h \alpha \partial_x + ib)u|^2 + ( \beta^{-1} \alpha(E - V_1) + \alpha^2( \partial_x \varphi)^2 + b_1^2)|u|^2, \\
u = e^{\varphi/h}v, \qquad v = (P(h) - E - i\ep)^{-1} \langle x \rangle^{-s} f, \qquad f \in L^2(\R).
\end{gathered}
\end{equation}
\begin{remark}
In the special case $\alpha, \beta = 1$, $V_1, b_0, b_1 = 0$,  and $\varphi = 0$, the functional \eqref{F} simplifies to the one used in \cite{lash24}.  
\end{remark} 

The goal is to specify a weight $w(x)$ having locally bounded variation, so that $d(wF)$ is bounded from below by
\begin{equation} \label{display lwr bd}
-h^{-2} w |e^{\varphi/h} f|^2 + 2\ep h^{-1} \beta^{-1} w \imag((\overline{h\alpha \partial_xu + ibu})u)+ \langle x \rangle^{-2s} (|u|^2 + |(h\alpha \partial_x + ib)u|^2),
\end{equation}
plus a remainder, see \eqref{lwr bd BV dwF3}. The condition \eqref{general inf} allows us to obtain the third term in \eqref{display lwr bd}. Our placement of $b$ and $b_1$ in \eqref{F} takes advantage of the characterization \eqref{D} of the domain, and respects the symmetry of the first order part of \eqref{mag schro oned}. In computing $d(wF)$, several convenient cancellations occur, namely \eqref{real parts vanish}, and terms involving $f$ and $V_0$ appear (see \eqref{expand dF}). The $V_0$ term is part of the remainder because it is left out of \eqref{F}, as its derivative may be irregular. 

The total remainder we incur upon computing $d(wF)$ can be thought of as $-\mu$ for the positive, finite measure $\mu$ given by \eqref{mu}. The most concerning feature of $\mu$ is its discrete part $\mu_d$, stemming from the point masses of $V_0$, $dV_1$, $db_1$, $d\alpha$, and $d\beta$. However, because these measures are assumed finite, $\mu_d$ has at most countably many point masses which are absolutely summable. Thus $-\mu_d$ can be compensated for by choosing a weight $w$ that has an extra ``regularization". As in \cite{lash24}, $w$ depends on an additional parameter $\eta > 0$ in such a way that $dw = dw_\eta$ includes a Gaussian approximation of $\mu_d$. We show in Appendix \ref{deal with point masses appendix} how, after integration, the desired estimate holds in the limit as $\eta \to 0^+$. 


In higher dimensions, there are strict requirements on the type of weight that can be used in the positive-commutator argument, see \cite[Section 2]{sh24}. This is one of the main obstacles to proving sharp resolvent estimates for $L^\infty$ coefficients in higher dimensions.  

\subsection{Applications} In Section \ref{applications section} we present several standard applications of our results when $V$, $b$, $1 - \alpha$, and $1 - \beta$ are independent of $h$ and have compact support. In this setting, Theorem \ref{Carleman mag schro 1D thm} is well known to have consequences for the imaginary parts of the scattering resonances for 
\begin{equation} \label{H}
H \defeq \beta(x) (- \partial_x (\alpha(x) \partial_x ) + b(x) D_x + D_xb(x)) + V(x).
\end{equation}
As in \cite{sz91}, we define the resonances of $H$ as the poles of the cutoff resolvent
\begin{equation} \label{cutoff resolv}
    \chi (H - \lambda^2)^{-1} \chi : L^2(\R) \to \mathcal{D}, \qquad \chi \in C^\infty_0(\R ; [0,1]),
\end{equation}
which continues meromorphically from $\imag \lambda \gg 1$ to the complex plane.
\begin{theorem} \label{unif resolv est thm}
Suppose $V(x)$, $b(x)$, $\alpha(x)$ and $\beta(x)$ (independent of $h$) are as in \eqref{decompose electric potential} through \eqref{inf alpha beta}. Suppose also that $V$, $b$, $1 - \alpha$, and $1 - \beta$  are supported in $[-R_0, R_0]$ for some $R_0 >0$. Fix $\lambda_0 > 0$, as well as  $\chi \in C_0^\infty(\mathbb R; [0,1])$ such that $\chi =1$ near $[-R_0, R_0]$. There exist $C, \theta_0 > 0$ so that for all $|\real \lambda| \ge \lambda_0$ and $|\imag \lambda| \le \theta_0$,
 \begin{equation} \label{unif resolv est}
      \|\chi (H- \lambda^2)^{-1} \chi\|_{H^{k_1} \to H^{k_2}} \le C|\real \lambda |^{k_2- k_1- 1}, \qquad k_1, k_2 \in \{0, \, 1\} \,(H^0 \defeq L^2(\R)),
 \end{equation}
 and 
 \begin{equation} \label{unif resolv est D}
      \|\chi (H - \lambda^2)^{-1} \chi\|_{L^2 \to \mathcal{D}} \le C|\real \lambda |,
 \end{equation}
 where $\mathcal{D}$ is equipped with the graph norm $\| u \|_{\mathcal{D}} \defeq (\|Hu\|^2_{L^2} + \|u\|^2_{L^2})^{1/2}$.
\end{theorem}

\begin{remark}
Since the hypotheses of Theorem \ref{unif resolv est thm} allow $\lambda_0$ to be any positive number, \eqref{unif resolv est D} precludes $\chi(H -\lambda^2)^{-1}\chi$ having a nonzero real resonance.
\end{remark}

In higher dimensions, high frequency resolvent bounds similar to \eqref{unif resolv est} and \eqref{unif resolv est D} were proved for Schr\"odinger operators with $L^\infty$ electric and magnetic potentials \cite{vo14b, mps20}. See also \cite{ccv13,ccv14}. An exponential high frequency resolvent bound for smooth potentials on non-compact Riemannian manifolds was recently proved in \cite{gr24}, extending results in \cite{cavo02}.

The existence of resonance free regions is a long-studied problem: \cite{ha82, zw87, hi99} treat the case of an electric potential only, with $V \in L^\infty_{\text{comp}}(\R)$, $V \in L_{\text{comp}}^1(\R)$, and $V$ exponentially decaying, respectively.  Several recent articles study resonance distribution for $h$-dependent Dirac masses \cite{sa16, dmw24, dama22} and, in higher dimensions, thin barriers \cite{gal19}. If $V \in L_{\text{comp}}^\infty(\R)$, one can use the classical Born series to show that the resonance free zone grows logarithmically in $|\real \lambda|$ \cite[Theorem 2.10]{dz}.  As far as the authors are aware, is it not known whether such a result holds for nontrivial magnetic potentials.

Our proof of Theorem \ref{unif resolv est thm} uses resolvent identities developed in \cite[Section 5]{vo14} for the Schr\"odinger operator with a compactly supported electric potential only, see also \cite{sh18, lash24}. But this strategy applies just as well in our setting, requiring only notational modification.

Estimates like \eqref{unif resolv est} yield integrability and decay for solutions to time-dependent equations involving $H$. 
\begin{corollary} \label{schro wave corollary}
Assume the hypotheses of Theorem \ref{unif resolv est thm} and in addition that $H$ has no resonance at $\lambda = 0$. Let $\mathbf{1}_{\ge 0}(H)$ be the orthogonal projection onto the nonnegative spectrum of $H$. There exist $C_1, \, C_2 > 0$ such that
\begin{gather}
\int_\R  \| \chi e^{i H t} \mathbf{1}_{\ge 0}(H) v \|^2_{L^2} dt \le C_1 \|v\|^2_{L^2}, \qquad v \in L^2(\R), \label{time integrability schro} \\
\| \chi \cos(\sqrt{|H|} t ) \mathbf{1}_{\ge 0}(H) \chi v \|_{L^2} \le e^{-C_2t}( \| v\|_{H^1} + \| \sqrt{|H|} \chi v\|_{L^2}), \qquad v \in H^1(\R),  \label{LED cosine wave} \\
\big \| \chi \frac{ \sin(\sqrt{|H|} t )}{\sqrt{|H|}} \mathbf{1}_{\ge 0}(H) \chi v  \big \|_{L^2} \le e^{-C_2t} \| v\|_{L^2}, \qquad v \in L^2(\R). \label{LED sine wave}
\end{gather}
\end{corollary}

\begin{remark}
In Appendix \ref{no zero resonance appendix}, we give simple examples of operators that do not have a resonance at zero. It seems challenging to find very general sufficient conditions on $H$'s coefficients that rule out a zero resonance.
\end{remark}

The proof of Corollary \ref{schro wave corollary} is a straightforward application of \eqref{unif resolv est} to Stone's formula, which represents the Schr\"odinger and wave propagators in terms of the limiting values of the resolvent, see \eqref{Stone's formula schro propagator} and \eqref{Stone's formula cosine propagator}. We expect Corollary \ref{schro wave corollary} can be improved in several ways, in particular, upgrading the $L^2$-norms on the right sides of \eqref{time integrability schro} and \eqref{LED cosine wave} to the $H^{1/2}$- and $H^1$- norm, respectively. See \cite[Section 2.3]{bgt04} and \cite[Section 7]{lash24}. Another question is whether similar integrability or decay holds for non-compactly supported coefficients. For this, one approach is to establish limiting absorption bounds for the weighted square of the resolvent \cite[Section 3]{cavo04}.

\medskip
\noindent{\textsc{Acknowledgements:}} We thank Georgi Vodev for helpful discussions, and the anonymous referee, whose thoughtful comments contributing to improving the paper. Both authors gratefully acknowledge support from NSF DMS-2204322. J. S. was also supported by a University of Dayton Research Council Seed Grant.\\

\section{Notation and review of BV} \label{bv review section}

To keep notation concise, for the rest of the article, we use ``prime" notation to denote differentiation with respect to $x$, e.g., $u' \defeq \partial_x u$.

In this section we collect several elementary properties of BV functions, which are used frequently in later sections. Proofs of Propostions \ref{prod rule bv prop} and \ref{chain rule bv prop} are given in \cite[Appendix B]{dash22}. The proof of Proposition \ref{ftc bv prop} appears in \cite[Section 2]{lash24}.

Suppose $f : \R \to \C$ is a function of locally bounded variation, in the sense that both the real and imaginary parts of $f$ are a difference of (not necessarily bounded) increasing functions. For all $x \in \R$, let 
\begin{equation*} 
 f^L(x) \defeq \lim_{\delta \to 0^+}f(x-\delta), \qquad  f^R(x) \defeq \lim_{\delta \to 0^+}f(x+\delta), \qquad f^A(x) \defeq (f^L(x) + f^R(x))/2,
\end{equation*}
be the left- and right-hand limits, and average value of $f$, respectively. Recall that $f$ is differentiable Lebesgue almost everywhere, so $f(x)= f^L(x) = f^R(x) = f^A(x)$ for almost all $x \in \R$.

We may decompose $f$ as 
\begin{equation*} 
f = f_{r, +} - f_{r, -} + i( f_{i,+} - f_{i,-}),
\end{equation*}
where the $f_{\sigma,\pm}$, $\sigma \in \{r, i\}$, are increasing functions on $\R$. Each $f^R_{\sigma,\pm}$ uniquely determines a regular Borel measure $\mu_{\sigma,\pm}$ on $\R$ satisfying $\mu_{\sigma, \pm}(x_1, x_2] = f^R_{\sigma, \pm}(x_2) -  f^R_{\sigma, \pm}(x_1)$, see \cite[Theorem 1.16]{fo}. We put
\begin{equation} \label{df}
    df \defeq \mu_{r, +} - \mu_{r, -} + i( \mu_{i,+} - \mu_{i,-}),
\end{equation}
 which is a complex measure when restricted to any bounded Borel subset. For any $a  < b$,
 
 \begin{equation} \label{ftc}
 \begin{gathered}
  \int_{(a,b]}df = f^R(b) - f^R(a),\\
  \int_{(a,b)}df = f^L(b) - f^R(a).
  \end{gathered}
 \end{equation}
 
 \begin{proposition}[product rule] \label{prod rule bv prop}
Let $f, \, g : \R \to \C$ be functions of locally bounded variation. Then
\begin{equation}\label{e:prod}
 d(fg) = f^A dg + g^A df
\end{equation}
as measures on a bounded Borel subset of $\R$.
\end{proposition}

\begin{proposition}[chain rule] \label{chain rule bv prop}
Let $f : \R \to \R$ be continuous and have locally bounded variation. Then, as measures on a bounded Borel set of $\R$,
\begin{equation} \label{chain rule continuous}
    d(e^f) = e^{f} df.
\end{equation}
\end{proposition}


\begin{proposition}[fundamental theorem of calculus] \label{ftc bv prop}
Let $\mu_{\sigma, \pm}$, $\sigma \in \{r,\,i\}$ be positive Borel measures on $\R$ which are finite on all bounded Borel subsets of $\R$. Suppose $u \in  \mathcal{D}'(\R)$ has distributional derivative equal to $\mu = \mu_{r, +} - \mu_{r, -} + i( \mu_{i,+} - \mu_{i,-})$. Then $u$ is of locally BV. For any $a \in \R$, $u$ differs by a constant from the right continuous, locally BV function
\begin{equation} \label{funct from meas}
f_\mu(x) \defeq \begin{cases}
\int_{[a, x]} d\mu & x \ge a, \\
-\int_{(x,a)} d\mu & x < a. 
\end{cases}
\end{equation}
\end{proposition}

\section{Proof of Theorem \ref{Carleman mag schro 1D thm} } \label{pf of main thm section}
In this section we prove Theorem \ref{Carleman mag schro 1D thm}. As discussed in Section \ref{intro section}, our proof is based on the energy method, which has long been used to establish limiting absorption resolvent estimates (see \cite{cavo02, da14, dash20, gash22a}). The starting point is the pointwise energy functional \eqref{F}.

Our calculations are facilitated by our characterization \eqref{D} of the domain of $P(h)$. In Lemma \ref{self-adjointness 1D lemma} in  Appendix \ref{self adj oned section}, we show $P(h)$ is self-adjoint with respect to \eqref{D}. Before examining the proof of Theorem \ref{Carleman mag schro 1D thm}, the reader may find it useful to first consult the proof Lemma \ref{self-adjointness 1D lemma}, to become acquainted with applying the properites of BV functions reviewed in Section \ref{bv review section}. 

\begin{proof}[Proof of Theorem \ref{Carleman mag schro 1D thm}]
In the following calculations, given $v \in \mathcal{D}$ as in \eqref{F}, we work with fixed representatives of $v$ and $h\alpha v' + ibv$ such that $v$ and $h\alpha v' + ibv$ are continuous and of locally bounded variation, respectively. This is permitted by \eqref{D}. By continuity, $v^A = v$ and $u^A = (e^{\varphi/h} v)^A = u$. By modifying  $h\alpha u' + ibu$ on a set of Lebesgue measure zero, we may suppose without loss of generality that $(h\alpha u' + ibu)^A = h\alpha u' + ibu$ too.

We compute $dF$, as a measure on a bounded Borel subset, using \eqref{e:prod}:
\begin{equation} \label{dF}
\begin{split}
dF &- |u|^2d(\beta^{-1} \alpha(E - V_1) + \alpha^2 (\varphi')^2 + b_1^2) \\
  &= 2 \real(( \overline{ h\alpha u' +i bu}) d(h\alpha u' + ibu)) + 2( \beta^{-1} \alpha(E - V_1) +\alpha^2 (\varphi')^2 + b_1^2) \real(\overline{u}' u) \\
  &= -2h^{-1} \real(( \overline{ h\alpha u' +i bu}) d(-h^2\alpha u' - ihbu)) \\
  &+ 2h^{-1} \real((\overline{h \alpha u' +i bu}) (\beta^{-1} (E - V_1) + \alpha(\varphi')^2)u)+ 2h^{-1} \alpha^{-1}b_1^2 \real((\overline{h \alpha u' +i bu})u).
\end{split}
\end{equation}
Now expand $d(-h^2\alpha u' - ihbu)$ by using \eqref{e:prod}, $u = e^{\varphi/h}v$, $u' = h^{-1} \varphi' u + e^{\varphi/h}v'$:
\begin{equation} \label{apply d}
\begin{split}
d(-h^2 \alpha u' - ihbu) &= d(e^{\varphi/h}(-h^2 \alpha v' - ihbv) -h \alpha \varphi'  u) \\
&= e^{\varphi/h} d(-h^2 \alpha v' - ihbv) -  e^{\varphi/h}\varphi' (2h \alpha v' + i b v) \\
 &- \alpha (\varphi')^2u - h \alpha^A u d(\varphi')  - h (\varphi')^A u d\alpha. 
 \end{split}
 \end{equation}
 Into the right side of \eqref{apply d} substitute $d(-h^2 \alpha v' - ihbv) = \beta^{-1} P(h)v - h bD_xv -\beta^{-1}vV$ and $2he^{\varphi/h} \alpha \varphi' v' = -2 \alpha (\varphi')^2 u + 2h \alpha \varphi' u'$, to get 
 \begin{equation} 
 \begin{split} \label{apply d pt 2}
d(-h^2 \alpha u' - ihbu) &=e^{\varphi/h} (\beta^{-1}P(h)v -hb D_x v - \beta^{-1}vV) - 2h \alpha \varphi' u' - ib \varphi'u\\
 &+\alpha(\varphi')^2u - h \alpha^A u d (\varphi')  - h (\varphi')^A u d\alpha \\
 &= e^{\varphi/h}  \beta^{-1} P(h)v -  \beta^{-1}uV + ihbu' - 2h \alpha \varphi' u' \\
 &+\alpha(\varphi')^2u - h \alpha^A u d(\varphi') -2ib \varphi'u - h (\varphi')^A u d\alpha.
\end{split}
\end{equation}
Inserting \eqref{apply d pt 2} into the right side of \eqref{dF} yields
\begin{equation*} 
\begin{split}
dF &- |u|^2d(\beta^{-1} \alpha(E - V_1) + \alpha^2 (\varphi')^2 + b_1^2) \\
  &= -2 h^{-1} \real(( \overline{h\alpha u' + i b u})(e^{\varphi/h} \beta^{-1} P(h)v - \beta^{-1} uV))\\
  &+ 2h^{-1} \real((\overline{h\alpha u' + i b u})(-ihbu' + 2h \alpha \varphi' u'
 -\alpha(\varphi')^2 u + h \alpha^A u d(\varphi') + 2ib\varphi'u + h (\varphi')^A u d\alpha))\\
  &+  2h^{-1} \real((\overline{h \alpha u' +i b u}) (\beta^{-1} (E - V_1) + \alpha (\varphi')^2)u))+ 2h^{-1} \alpha^{-1}b_1^2 \real((\overline{h \alpha u' +i bu})u).
  \end{split}
\end{equation*}
Now we simplify terms where convenient. In particular, 
\begin{equation} \label{real parts vanish}
\begin{split}
-2&h^{-1} \real((\overline{h\alpha u' + ibu})(ihbu')) + 2h^{-1} \alpha^{-1}b_1^2 \real((\overline{h \alpha u' +i bu})u)\\
&= -2h^{-1} \real((\overline{h\alpha u' + ibu})(ib\alpha^{-1}(h \alpha u' + ibu) + \alpha^{-1} b^2u) + 2h^{-1} \alpha^{-1}b_1^2 \real((\overline{h \alpha u' +i bu})u\\
&=  2h^{-1} \alpha^{-1} (b^2_1 - b^2) \real((\overline{h \alpha u' +i b u})u).
\end{split}
\end{equation}
We also add and subtract the term $-2 h^{-1} \beta^{-1} \real (( \overline{h\alpha u' + i b u})i \ep u) = 2 \ep h^{-1} \beta^{-1} \imag(( \overline{h\alpha u' + i b u})u)$. The output is
\begin{equation*} 
\begin{split}
 dF  &- |u|^2d(\beta^{-1} \alpha(E - V_1) + \alpha^2 (\varphi')^2 + b_1^2) \\
 &=-2h^{-1} \real( ( \overline{h\alpha u' + i b u})e^{\varphi/h} \beta^{-1} (P(h) -E - i\ep)v) + 2 \ep h^{-1} \beta^{-1} \imag(( \overline{h\alpha u' + i b u})u)  \\
  &+4 h \alpha^2 \varphi' |u'|^2 - 4 \alpha b \varphi' \imag( \overline{u}' u)\\
  &+ 2h^{-1}\real((\overline{h\alpha u' + i b u})(h \alpha^A u d(\varphi') +2ib\varphi'u + h(\varphi')^A u d\alpha ) \\
  &+2 h^{-1} \real((\overline{h \alpha u' + ibu})u)( \alpha^{-1} (b_1^2 - b^2)  + \beta^{-1} V_0). 
  \end{split}
  \end{equation*}
  Using $v = (P(h) - E - i\ep)^{-1} \langle x \rangle^{-s} f$ and the identity
  \begin{equation*}
\begin{split}
 - 4 \alpha b \varphi' \imag( \overline{u}' u) &+ 2h^{-1}\real((\overline{h\alpha u' + i b u}) 2ib\varphi'u) \\
 &= -4 h^{-1} b^2 \varphi' |u|^2 - 8h^{-1} b \varphi' \imag((\overline{h\alpha u' + i b u})u),
 \end{split}
\end{equation*}
we arrive at 
\begin{equation} \label{expand dF}
\begin{split}
dF  &- |u|^2d(\beta^{-1} \alpha(E - V_1) + \alpha^2 (\varphi')^2 + b_1^2) \\
  &= -2h^{-1} \real( ( \overline{h\alpha u' + i b u})e^{\varphi/h} \beta^{-1} \langle x\rangle^{-s} f) + 2 \ep h^{-1} \beta^{-1} \imag(( \overline{h\alpha u' + i b u})u)  \\
  &+4 h \alpha^2 \varphi' |u'|^2 -4 h^{-1} b^2 \varphi' |u|^2 - 8h^{-1} b \varphi' \imag((\overline{h\alpha u' + i b u})u)  \\
    &+2 h^{-1} \real((\overline{h \alpha u' + ibu})u)( h \alpha^A d(\varphi') + h(\varphi')^A d\alpha + \alpha^{-1} (b_1^2 - b^2)  + \beta^{-1} V_0).
\end{split}
\end{equation}

Momentarily, we shall define a continuous weight $w(x)$ which is bounded, has locally BV, $dw \ge 0$, and $w \varphi' \ge 0$. For such a $w$, \eqref{e:prod} and \eqref{expand dF} imply
\begin{equation*}
\begin{split}
d&(wF)\\
 &= F^A dw + w dF \\
&=  (|h \alpha u' + ibu|^2)^Adw + (\beta^{-1} \alpha(E - V_1) + \alpha^2 (\varphi')^2 +b_1^2)^A|u|^2dw \\ 
& -2 h^{-1} w \real((\overline{h \alpha u' +i bu}) e^{\varphi/h} \langle x \rangle^{-s} f)  + 2 \ep h^{-1} \beta^{-1} w \imag ((\overline{h \alpha u' +i bu}) u)\\
 &+4 h \alpha^2  \varphi' w |u'|^2 -4h^{-1}b^2 \varphi'w|u|^2 - 8h^{-1} b \varphi' w \imag((\overline{h\alpha u' + i b u})u) \\
    &+2 h^{-1}w \real((\overline{h \alpha u' + ibu})u) (h \alpha^A d(\varphi') + h(\varphi')^A d\alpha + \alpha^{-1} (b_1^2 - b^2)  + \beta^{-1} V_0) \\
    & + w |u|^2d(\beta^{-1}\alpha(E - V_1) + \alpha^2 (\varphi')^2 + b_1^2). 
    \end{split}
    \end{equation*}
 To find a lower bound for $d(wF)$, we discard the term $4 h \alpha^2  \varphi' w |u'|^2$ since $\varphi' w \ge 0$, use \eqref{tau}, and bound from below some terms involving $h \alpha u' +i bu$. We find
    \begin{equation}  \label{dwF}
    \begin{split}
   d(wF) & \ge  \tau|u|^2 dw  -  |w| |u|^2(4h^{-1}b^2 |\varphi'| + |d(\beta^{-1} \alpha(E - V_1) + \alpha^2 (\varphi')^2 + b_1^2)|)  \\
    &+ (|h \alpha u' + ibu|^2)^A dw - |w| |h \alpha u' + ibu|^2 \langle x \rangle^{-2s} \\
    &- 2|w||u||h \alpha u' + ibu| (h^{-1}| \alpha^{-1} (b_1^2 - b^2)  + \beta^{-1} V_0| +  |\alpha^A d(\varphi')| + |(\varphi')^A d\alpha| + 4h^{-1} |\varphi' b|)\\
&- h^{-2} |w| |e^{\varphi/h} f|^2 + 2 \ep h^{-1} \beta^{-1} w \imag((\overline{h \alpha u' +i bu}) u).
\end{split}
\end{equation}

Next, let $\mu$ be the nonnegative, finite measure 
\begin{equation} \label{mu}
\begin{split}
\mu = \mu(h) &\defeq  h^{-1}| \alpha^{-1} (b_1^2 - b^2)  + \beta^{-1} V_0| +  |\alpha^A d(\varphi')| + |(\varphi')^A d\alpha| + 4h^{-1} |\varphi'|(b^2 + |b|)  \\ &+|d(\beta^{-1} \alpha(E - V_1) + \alpha^2 (\varphi')^2 + b_1^2)|.
\end{split}
\end{equation}
Note that while $b^2 + |b|$ is not necessarily a finite measure, $|\varphi'|(b^2 + |b|)$ is since $\varphi'$ has compact support; $b_1^2 - b^2 = -2b_0 b_1 + b^2_0$ is a finite measure because $b_0 \in L^1(\R) \cap L^2(\R)$.

Before constructing $w$, we make several observations about $\mu$. Let us decompose
\begin{equation*}
\mu = \mu_{c} + \mu_{d}
\end{equation*}
into its continuous and discrete parts. Since $V_0$ is a finite measure, and since $V_1$, $b_1$, $\alpha$ and $\beta$ have bounded variation, $\mu_{d}$ consists of at most countably many point masses, which are absolutely summable. Let $\{x_j\} \subseteq \R$ be an enumeration of the point masses, and put $\mu_{j} = \mu_{d}(x_j)$. Then 
\begin{equation*}
\mu = \mu_{c} + \sum_j \mu_{j} \delta_{x_j},
\end{equation*}
where $\delta_{x_j}$ denotes the Dirac measure concentrated at $x_j$.

We are now prepared to define $w$. We use the family of weights,
\begin{equation} \label{weta}
w(x) = w_\eta(x) =  -e^{q_{1, \eta}(x)}(e^{q_{2}(x)}-1) \mathbf{1}_{(-\infty, 0)}(x)  + e^{q_{1, \eta}(x)}  (e^{q_{2}(x)}-1) \mathbf{1}_{(0, \infty)}(x), 
 \end{equation}
depending on the additional parameter $\eta > 0$, where 
\begin{equation} \label{q1 BV}
q_{1, \eta}(x) =   \sgn(x) \int_{0}^x  \mu_c + \pi^{-1/2} \eta^{-1}  \sum_{x_j \neq 0 }  W_j e^{-((x'-x_j)/\eta)^2} dx'.
\end{equation}
and
\begin{equation} \label{q2 BV}
   q_{2}(x)  \defeq  \sgn(x) \kappa \int_{0}^x  \langle x' \rangle^{-2s} dx' .
\end{equation}
Here,
\begin{equation*}
\sgn(x) = \begin{cases}
1 & x \ge 0, \\
-1 & x < 0,
\end{cases}
\end{equation*}
while $\kappa \ge 1$ and summable $W_j \ge 0$ will be chosen in due course, independent of $\eta$. 

Note that $w(0) = 0$, implies the product $w \delta_0 = 0$, which we make use of in the ensuing estimates. In addition, observe that $\sup |w_\eta|$ is bounded from above independent of $\eta$ since for all $\eta > 0$
\begin{equation*}
|q_{1, \eta}(x)| \le \mu_c(\R) + \pi^{-1/2} \sum_{x_j \neq 0} W_j \int^\infty_0 e^{-(x')^2}dx' \le \mu_c(\R) + \frac{1}{2} \sum_{x_j \neq 0} W_j.
\end{equation*}
For later use, we put
\begin{equation} \label{Cw}
 C_w \defeq \sup_{x \in \R, \, \eta > 0} |w_\eta(x)|.
 \end{equation}

By \eqref{e:prod} and \eqref{chain rule continuous}, 
\begin{equation} \label{dw}
\begin{split}
    dw &= \kappa \langle x \rangle^{-2s} e^{q_{1, \eta} + q_{2}} + |w|(\mu_c + \pi^{-1/2} \eta^{-1} \sum_{x_j \neq 0}   W_j e^{-((x-x_j)/\eta)^2}).  \\
  \end{split}
\end{equation}
Substituting \eqref{dw} into the right side \eqref{dwF}, we conclude, for $\gamma_j > 0$ to be chosen,
\begin{equation} \label{lwr bd BV dwF2}
\begin{split}
d(wF) & \ge  |u|^2(\kappa \tau  \langle x \rangle^{-2s} + (\kappa \tau -1)|w| \mu_c) + |h \alpha u' + ibu|^2((\kappa -1) \langle x \rangle^{-2s} + (\kappa - 1)|w| \mu_c)   \\ 
&+ |w||u|^2 \sum_{x_j \neq 0}  ( \tau \pi^{-1/2} \eta^{-1} W_j  e^{-((x-x_j)/\eta)^2} - (1 + \gamma_j^{-1}) \mu_j \delta_{x_j}\big) \\
&+  |w| |h \alpha u' + ibu|^2\sum_{x_j \neq 0}  (\pi^{-1/2} \eta^{-1} W_j e^{-((x-x_j)/\eta)^2} - \gamma_j \mu_j \delta_{x_j}\big) \\
&- h^{-2} |w| |e^{\varphi/h} f|^2 + 2 \ep h^{-1} \beta^{-1} w \imag((\overline{h \alpha u' +i bu}) u).
\end{split}
\end{equation}
Note the $\gamma_j$ arise from using Young's inequality at the point masses of $\mu$:
\begin{equation*}
2|w||u||h \alpha u' + ibu| \mu_j \delta_{x_j} \ge ( \gamma_j^{-1} |u|^2 + \gamma_j |h \alpha u' + ibu|^2) |w| \mu_j \delta_{x_j}.
\end{equation*}
The summations appearing in lines two and three do not include the point mass that $\mu$ may have at zero, since $w$ vanishes there. 

Now, fix $\kappa = \kappa(h) \defeq \max(2, 1/\tau(h))$ so that \eqref{lwr bd BV dwF2} implies 

\begin{equation} \label{lwr bd BV dwF3}
\begin{split}
d(wF) & \ge  \langle x \rangle^{-2s} |u|^2 + \langle x \rangle^{-2s}  |h\alpha u' + ibu|^2 \\
&+   |w||u|^2 \sum_{x_j \neq 0} ( \tau \pi^{-1/2} \eta^{-1} W_j   e^{-((x-x_j)/\eta)^2} - (1 + \gamma_j^{-1})  \mu_j \delta_{x_j} ) \\
&+ |w| |h \alpha u' + ibu|^2\sum_{x_j \neq 0} ( \pi^{-1/2} \eta^{-1} W_j e^{-((x-x_j)/\eta)^2} - \gamma_j \mu_j \delta_{x_j} ) \\
 &- h^{-2} |w| |e^{\varphi/h} f|^2 + 2 \ep h^{-1} \beta^{-1} w \imag((\overline{h \alpha u' +i bu}) u).
\end{split}
\end{equation}

To proceed, we integrate \eqref{lwr bd BV dwF3}. because $F(x) \in L^1(\R)$ and is continuous off of a countable set, there are sequences $\{a^\pm_n\}_{n=1}^\infty$ tending to $\pm \infty$, along which $F(a^\pm_n) \to 0$. Thus, after integrating \eqref{lwr bd BV dwF3} over $(a^-_n, a^+_n]$ and sending $n \to \infty$, the left side of \eqref{lwr bd BV dwF3} becomes zero. Therefore

\begin{equation} \label{just after integrate}
\begin{split}
\int& \langle x \rangle^{-2s} (|u|^2 + |(h\alpha \partial_x + ib)u|^2) dx \\
&+ \sum_{x_j \neq 0}  \int |w||u|^2 ( \tau \pi^{-1/2} \eta^{-1} W_j   e^{-((x-x_j)/\eta)^2} dx- (1 + \gamma_j^{-1}) \mu_j \delta_{x_j} ) \\
&+ \sum_{x_j \neq 0}  \int |w||h \alpha u' + ibu|^2 ( \pi^{-1/2} \eta^{-1} W_j   e^{-((x-x_j)/\eta)^2} dx-  \gamma_j \mu_j \delta_{x_j} ) \\
&\le h^{-2} C_w \int |e^{\varphi/h}f|^2 dx + \tfrac{|\ep| C_w}{h \inf \beta }  \int  |h \alpha u' +i b u|^2 +  |u|^2 dx,
\end{split}
\end{equation}
where we have use \eqref{Cw}.

Our goal is to show that, upon choosing the $W_j$ and $\gamma_j$ appropriately, and in the limit $\eta \to 0^+$, the resulting quantities in lines two and three of \eqref{just after integrate} are nonnegative, yielding the simpler estimate 
\begin{equation} \label{pre penult est} 
\begin{split}
\int& \langle x \rangle^{-2s} ( |u|^2 + |h\alpha u' + ibu|^2) dx \\
&\le  h^{-2} C_w \int |e^{\varphi/h}f|^2 dx +  \tfrac{|\ep| C_w}{h \inf \beta }  \int  |h \alpha u' +i bu|^2 +  |u|^2 dx.
\end{split}
\end{equation}
This calculation is elementary but tedious. The details are given in Appendix \ref{deal with point masses appendix}.

Our final task is to appropriately estimate the term involving $|(h \alpha \partial_x +i b)u|^2$ in the second line of \eqref{pre penult est}. We start with  $u' = h^{-1} \varphi' u + e^{\varphi/h} v'$ and 
\begin{equation} \label{h alpha u prime plus b square} 
\begin{split}
\int  |h \alpha u' +i bu|^2 dx &= \int |e^{\varphi/h} h \alpha v' + (\alpha \varphi' + ib)u|^2 dx \\
& \le  2 e^{2\sup \varphi /h} \big( \sup \alpha  \int  \alpha |h v'|^2 dx +   2( \|\alpha \varphi'\|^2_{L^\infty} + \|b_1\|^2_{L^\infty}) \int |v|^2 dx \\
&+  \|b_0\|^2_{L^2} \| v \|^2_{L^\infty} \big).
\end{split}
\end{equation}
The well known bound $\|v\|^2_{L^\infty} \le \|v\|_{L^2} \|v'\|_{L^2}$ implies
\begin{equation} \label{peter paul Linfty}
 \|v\|^2_{L^\infty} \le  \tfrac{1}{2 h^2 \gamma \inf \alpha} \|v \|^2_{L^2} +  \frac{\gamma}{2} \int \alpha |hv'|^2, \qquad \gamma > 0.
\end{equation}
Thus by \eqref{h alpha u prime plus b square},
\begin{equation*}
\int  |(h \alpha \partial_x +i b)u|^2 dx \le  C(h) \int  \alpha |h v'|^2 + |v|^2 dx 
\end{equation*}
for some $C(h) > 0$ with the dependencies as described in the statement of Theorem \ref{Carleman mag schro 1D thm}. We continue to use this constant although its precise value may change from line to line. 

 Estimate $\int  \alpha |h v'|^2 dx$ using \eqref{prod to form}:
\begin{equation} \label{h alpha v prime square}
\begin{split}
\int  \alpha |h v'|^2 &=   \real \langle (P(h) - E -i\ep)v ,v \rangle_{L^2(\beta^{-1}dx)} + 2h \imag \int b \overline{v}'v + \int \beta^{-1} |v|^2 (E -V) \\
& \le  \int (\tfrac{1}{2 \inf \beta} |f|^2 + ( \tfrac{1}{2 \inf \beta} + \|E - V_1\|_{L^\infty}  + \tfrac{4 \|b_1\|^2_{L^\infty}}{\inf \alpha} ) |v|^2)dx  +   \tfrac{1}{2} \int \alpha |hv'|^2dx \\
&+ \big( \tfrac{ \|V_0\| }{\inf \beta} + \tfrac{4\|b_0\|^2_{L^2}}{\inf \alpha} \big) \|v\|^2_{L^{\infty}}.
\end{split}
\end{equation}
Using \eqref{peter paul Linfty} once more to bound the factor $\|v\|^2_{L^\infty}$ in the last line of \eqref{h alpha v prime square}, we get a bound for $\int \alpha |hv'|^2$ in terms of integrals of $|f|^2$ and $|v|^2$. Combining this with \eqref{pre penult est}, \eqref{h alpha u prime plus b square}, and $|\ep| \le \ep_0$ implies 
\begin{equation} \label{ult est}
 \int \langle x \rangle^{-2s}(|u|^2 + |(h\alpha \partial_x + ib)u|^2) dx  \le C(h) \int |f|^2 + |\ep| |v|^2 dx,
 \end{equation}
which concludes the proof of \eqref{weighted est mag schro}.\\
 \end{proof}
 
 \section{Proofs of Corollaries \ref{ext resolv est corollary} and \ref{exp resolv est corollary}}
 
 In this section, we show how Corollaries \ref{ext resolv est corollary} and \ref{exp resolv est corollary} follow from the proof of Theorem \ref{Carleman mag schro 1D thm}.
 
 \begin{proof}[Proof of \eqref{ext resolv est mag schro}]
Set $v = (P(h) - E - i\ep)^{-1} \langle x \rangle^{-s} \mathbf{1}_{> R} f$ for $f \in L^2(\R)$ arbitrary.  Start from the estimate \eqref{dwF} in the proof of Theorem \ref{Carleman mag schro 1D thm}. Only now we take $\varphi = 0$ and ask that our weight $w_\eta$ vanishes on $[-R, R]$. Thus we have the simpler lower bound 
\begin{equation} \label{lwr bd dwF ext}
\begin{split}
d&(wF) \\
 & \ge  |v|^2( \tau_1 dw  -  |w| \mu) \\
    &+ |h \alpha v' + ibv|^2 dw - |w| |h \alpha v' + ibv|^2 \langle x \rangle^{-2s}  \\
&- h^{-2} w |f|^2 + 2 \ep h^{-1} \beta^{-1} w \imag((\overline{h \alpha v' +i bv}) v),
\end{split}
\end{equation}
where this time the measure $\mu$ is only
\begin{equation*}
\mu =  |d(\beta^{-1}\alpha(E - V_1) + b^2_1)|.
\end{equation*}
Note that under the hypotheses of Corollary \ref{ext resolv est corollary}, $\mu(\R)$ is bounded independently of $h$. Also, there is no need to use the average value of $h \alpha v' + ibv$ in \eqref{lwr bd dwF ext}. Indeed, $h \alpha v' + ibv$ is locally absolutely continuous on the support of $w$, since $V_0 = 0$ there.

Let $\{x_j\}$ be an enumeration of the point masses of $\mu$ in $(-\infty, -R) \cup (R, \infty)$. We take $w_\eta$ similar to \eqref{weta}, though adjusted so it is zero on $[-R, R]$:

 \begin{equation} \label{weta ext}
 w_\eta(x) = \begin{cases} \sgn(x) e^{q_{1, \eta}(x)} (e^{q_2(x)} -1)  & |x| > R, \\
0 & |x| \le R. \end{cases} 
\end{equation}
where 
\begin{equation} \label{q1 BV ext}
q_{1, \eta}(x) =   \sgn(x) \int_{\sgn(x)R}^x  \mu_c + \pi^{-1/2} \eta^{-1}  \sum_{\sgn(x) x_j > R }  W_j e^{-((x'-x_j)/\eta)^2} dx'.
\end{equation}
and
\begin{equation} \label{q2 BV}
   q_{2}(x)  \defeq  \sgn(x) \kappa \int_{\sgn(x)R}^x  \langle x' \rangle^{-2s} dx' ,
\end{equation}
for suitable $\kappa > 0$ and $W_j$ to be chosen. This yields
\begin{equation} \label{dweta ext}
\begin{split}
dw & = \kappa \langle x \rangle^{-2s} e^{q_{1, \eta} + q_2} (\mathbf{1}^{A}_{(-\infty, -R)} + \mathbf{1}^{A}_{(R, \infty)}) \\
&+ |w| \mathbf{1}_{(-\infty, -R)} (\mu_c + \pi^{-1/2} \eta^{-1} \sum_{x_j < -R} W_j e^{-((x-x_j)/\eta)^2}) \\
&+  |w| \mathbf{1}_{(R, \infty)} (\mu_c + \pi^{-1/2} \eta^{-1} \sum_{x_j > R} W_j e^{-((x-x_j)/\eta)^2})
\end{split}
\end{equation}

Now plug \eqref{weta ext} and \eqref{dweta ext} into \eqref{lwr bd dwF ext} and follow the steps in the proof of Theorem \ref{Carleman mag schro 1D thm}, beginning from \eqref{lwr bd BV dwF3}. That is, fix $\kappa > 0$ large enough (independent of $h$) to control the portions of lines two and three of \eqref{lwr bd dwF ext} which are absolutely continuous. Handling the point masses $\mu_j \delta_{x_j}$ is simpler in this situation, since $\sum_j \mu_j$ is bounded independent of $h$, and since $\mu$ appears in line two of \eqref{lwr bd dwF ext} only. We again refer the reader to the steps of Appendix \ref{deal with point masses appendix}. Then, by sending $\eta \to 0^+$, we get
\begin{equation} \label{pre penult 1D BV ext est} 
\begin{split}
\int_{\R \setminus [-R, R]}  &\langle x \rangle^{-2s}(|v|^2 + |h\alpha v' + ibv|^2)dx \\
 &\le  C_0 h^{-2}  \int_{\R \setminus [-R, R]}  |f|^2 dx \\
 &+C_0 |\ep| h^{-1} \big(  \int_{-\infty}^{-R}  |v|^2 +|(e^{q_{2}} - 1)(h \alpha v' + ibv)|^2 dx \\
  &+ \int_{R}^{\infty} |v|^2 +|(e^{q_{2}} - 1)(h \alpha v' + ibv)|^2 dx \big),
 \end{split}
\end{equation}
for some $C_0 > 0$ independent of $h$, whose value may change from line to line.

As before, we need to estimate the terms involving $\ep$ on the right side of \eqref{pre penult 1D BV ext est}. For the terms in which $|h \alpha v' + ibu|^2$ appears, we can prepare to integrate by parts, since as noted above $h \alpha v' + ibv$ is locally absolutely continuous on our present domains of integration. For instance,
\begin{equation} \label{prepare to integrate by parts}
\begin{split}
 \int_{R}^\infty (e^{q_{2, R}} - 1)^2 |h \alpha v' + ibv|^2dx &\le \sup \alpha  \int_{R}^\infty (e^{q_{2, R}} - 1)^2 \alpha^{-1} |h \alpha v' + ibv|^2dx \\
 &= \sup \alpha \real \big( \int_{R}^\infty (e^{q_{2}} - 1)^2 h\overline{v}' (h\alpha v' + ibv) dx \\
 &-i  \int_{R}^\infty (e^{q_{2, R}} - 1)^2 \alpha^{-1} b  \overline{v} (h\alpha v' + ibv) dx \big) \\
 & \le \sup \alpha \real \int_{R}^\infty (e^{q_{2}} - 1)^2 h\overline{v}' (h\alpha v' + ibv) dx \\
 &+ C_0 \int_{R}^\infty |v|^2 dx + \tfrac{1}{4} \int_{R}^\infty (e^{q_{2}} - 1)^2 |h \alpha v' + ibv|^2dx.
 \end{split}
\end{equation}
Note we estimated the term  $\imag \int_{R}^\infty (e^{q_{2, R}} - 1)^2 \alpha^{-1} b  \overline{v} (h\alpha v' + ibu) dx$ using Young's inequality, along with the fact that $|e^{q_{2}} -1| \le C_0$. We shall estimate in this manner several more terms that arise.
Focusing now on line four of \eqref{prepare to integrate by parts}, and integrating by parts:
\begin{equation} \label{integrate by parts ext est}
\begin{split}
\real  \int_{R}^\infty& (e^{q_{2}} - 1)^2 h\overline{v}' (h\alpha v' + ibv) dx \\ 
 &= \real \int_{R}^\infty (e^{q_{2}} - 1)^2 \overline{v} (-h^2 \alpha v' - ihbv)' - 2 \kappa \langle x \rangle^{-2s} e^{q_{2}} (e^{q_{2}} - 1) h \overline{v}(h \alpha v' + ibv)  dx \\
 &= \real \int_{R}^\infty (e^{q_{2}} - 1)^2 \beta^{-1} \overline{v} (P(h) - E - i\ep)v dx \\
 &- h \imag \int_{R}^\infty (e^{q_{2}} - 1)^2 b \overline{v} v'   dx +  \int_{R}^\infty (e^{q_{2}} - 1)^2 \beta^{-1}  (E - V_1)|v|^2 dx \\
 & - 2 \kappa \int_{R}^\infty \langle x \rangle^{-2s} e^{q_{2}} (e^{q_{2}} - 1) h \overline{v}(h \alpha v' + ibv)  dx\\
 &= \real \int_{R}^\infty (e^{q_{2}} - 1)^2 \beta^{-1} \overline{v} (P(h) - E - i\ep)v dx \\
 &- \imag \int_{R}^\infty (e^{q_{2}} - 1)^2 \alpha^{-1} b \overline{v} (h\alpha v' + ibv)   dx \\
 &+ \int_{R}^\infty \alpha^{-1} (e^{q_{2}} - 1)^2 b^2 |v|^2 dx +  \int_{R}^\infty (e^{q_{2}} - 1)^2 \beta^{-1}  (E - V_1)|v|^2 dx \\
 &- 2 \kappa  \int_{R}^\infty \langle x \rangle^{-2s} e^{q_{2}} (e^{q_{2}} - 1) h \overline{v}(h \alpha v' + ibv)  dx \\
 &\le C_0 \int_{R}^\infty |f|^2 + |v|^2 dx + \tfrac{1}{4} \int_{R}^\infty (e^{q_{2}} - 1)^2 |h \alpha v' + ibv|^2dx.
\end{split}
\end{equation}
From \eqref{prepare to integrate by parts} and \eqref{integrate by parts ext est} we deduce,
\begin{equation*}
\int_{R}^\infty (e^{q_{2, R}} - 1)^2  |h \alpha v' + ibv|^2 \le C_0 \int_R^\infty |f|^2 +  |v|^2 dx.
\end{equation*}
Substituting this into the right side of \eqref{pre penult 1D BV ext est} implies 
\begin{equation} \label{penult 1D BV ext est}
\begin{split}
\int_{\R \setminus [-R, R]}  &\langle x \rangle^{-2s}(|v|^2 + |h\alpha v' + ibv|^2)dx \\
 &\le C_0h^{-2}  \int_{\R \setminus [-R, R]}  |f|dx +C_0|\ep| h^{-1} \int_{\R \setminus [-R, R]} |v|^2 dx \\
 &+ C_0 |\ep| h^{-1} \int_{-\infty}^{-R}|(e^{q_{2}} - 1)(h \alpha v' + ibv)|^2 dx. 
 \end{split}
\end{equation}
Line three of \eqref{penult 1D BV ext est} may be estimated in a similar way, so that 
 \begin{equation}  \label{ult 1D BV ext est}
\begin{split}
\int_{\R \setminus [-R, R]}  &\langle x \rangle^{-2s}(|v|^2 + |h\alpha v' + ibv|^2)dx \\
 &\le  C_0h^{-2}  \int_{\R \setminus [-R, R]}  |f|dx +C_0|\ep| h^{-1} \int_{\R \setminus [-R,R]} |v|^2 \beta^{-1} dx.
 \end{split}
\end{equation}

Then, because $(P(h) - E - i\ep)v =  \langle x \rangle^{-s}  \mathbf{1}_{> R} f$,
\begin{equation*}
\begin{split}
C_0|\ep| h^{-1} \int_{\R \setminus [-R,R]} |v|^2 dx &\le C_0 |\ep| h^{-1} \sup \beta \int_{\R \setminus [-R,R]} |v|^2 \beta^{-1} dx \\
&= C_0 \sgn(\ep) h^{-1} \imag \langle (P(h) - E - i\ep) v, v  \rangle_{L^2 (\beta^{-1} dx)} \\
&\le C_0h^{-2}  \int_{\R \setminus [-R,R]} |f|^2 dx + \tfrac{1}{2} \int_{\R \setminus [-R,R]} \langle x \rangle^{-2s} |v|^2 dx.
\end{split}
\end{equation*}
The term $\tfrac{1}{2} \int_{\R \setminus [-R,R]} \langle x \rangle^{-2s} |v|^2 dx$ may be absorbed into the left side of \eqref{ult 1D BV ext est}. Therefore,
 \begin{equation*}  
\begin{split}
\int_{\R \setminus [-R, R]}  &\langle x \rangle^{-2s}(|v|^2 + |h\alpha v' + ibv|^2)dx \le  \frac{C_0}{h^2}  \int_{\R \setminus [-R, R]}  |f|^2dx,
 \end{split}
\end{equation*}
finishing the proof of \eqref{ext resolv est mag schro}. \\
 \end{proof} 
 
 \begin{proof}[Proof of \eqref{resolv est  mag schro}]
Under the hypotheses  of Corollary \ref{exp resolv est corollary}, in the proof of Theorem \ref{Carleman mag schro 1D thm}, we may take $\varphi$ (as in \eqref{varphi}) and thus $\tau$ (as in \eqref{tau}) independent of $h$.  Moreover, by the supposed uniformity of the coefficients with respect $h$, the measure $\mu$ as in \eqref{mu} obeys $\mu(\R) \le C_1h^{-1}$ for some $C_1 > 0$ independent of $h$. We reuse this constant below though its precise value may change.  

In particular, we attain \eqref{fix Wj} in Appendix \ref{deal with point masses appendix} by setting $W_j = M \mu_j$, with $M = \max(2, 8 \tau^{-1})$ (which is independent of $h$), and $\mu_j \delta_{x_j}$ the nonzero point masses of $\mu$. Thus, under these assumptions, $C_w$ as in \eqref{Cw}, and the other constants in front of integrals on the right sides of \eqref{pre penult est} through \eqref{h alpha v prime square}, are bounded from above by a single constant of the form $e^{C_1/h}$, giving
\begin{equation} \label{ult est exp}
 \int \langle x \rangle^{-2s}|e^{\varphi/h} v|^2 dx  \le e^{C_1/h} \int |f|^2 + |\ep| |v|^2 dx.
\end{equation}
Recalling $v = (P(h) - E(h) - i\ep)^{-1} \langle x \rangle^{-s} f$, the $|u|^2$-term on the right side has the bound
\begin{equation} \label{ep int u square exp}
\begin{split}
 \ep e^{C_1/h} \int |v|^2dx &= |\ep| e^{C_1/h} \sup \beta \int |v|^2 \beta^{-1} dx  \\
 &\le e^{C_1/h} \sgn(\ep) \imag \langle (P(h) - E - i\ep)v,v  \rangle_{L^2(\beta^{-1}dx)}  \\
&\le e^{C_1/h} \int |f|^2 dx + \tfrac{1}{2} \int \langle x \rangle^{-2s} |v|^2 dx. \end{split}
\end{equation} 
We may absorb the second term in line three of \eqref{ep int u square exp} into the right side of \eqref{ult est exp}, implying \eqref{resolv est mag schro}.\\
 \end{proof} 
  
 \section{Applications} \label{applications section}
 
 \subsection{Uniform resolvent estimates and resonance free strip} \label{uniform resolv est subssection}
 
We prove Theorem \ref{unif resolv est thm} as an application of Theorem \ref{Carleman mag schro 1D thm}. As described in Section \ref{intro section}, we are concerned with
\begin{equation} \label{H}
H \defeq \beta(x) (- \partial_x (\alpha(x) \partial_x ) + b(x) D_x + D_xb(x)) + V(x),
\end{equation}
We suppose the coefficients satisfy \eqref{decompose electric potential} through \eqref{inf alpha beta}, although now they are independent of the semiclassical parameter, and $V_0$, $V_1$, $b_0$, $b_1$, $1 - \alpha$, and $1 - \beta$ have support in $[-R_0, R_0]$.

In this situation, $H$ is a \textit{black box Hamiltonian} in the sense of Sj\"ostrand and Zworski \cite{sz91}, as defined in \cite[Definition 4.1]{dz}. More precisely, in our setting this means the following. First, if  $u \in \mathcal{D}$, then $u|_{\mathbb R\setminus [-R_0,R_0]} \in H^2(\mathbb R\setminus [-R_0,R_0])$.  Second, for any $u \in \mathcal{D}$, we have $(Hu)|_{\mathbb R\setminus [-R_0,R_0]}  = - (u|_{\mathbb R\setminus [-R_0,R_0]})''$. Third, any $u \in H^2(\mathbb R)$ which vanishes on a neighborhood of $[-R_0,R_0]$ is also in $\mathcal{D}$. Fourth, $\textbf{1}_{[-R_0,R_0]} (H+i)^{-1}$ is compact on $\mathcal H$; this last condition follows from the fact that $\mathcal{D} \subseteq H^1(\R)$.

Then, by the analytic Fredholm theorem (see \cite[Theorem 4.4]{dz}), we have the following. In $\imag \lambda > 0$, the resolvent $(H - \lambda^2)^{-1}$ is meromorphic $L^2(\R) \to \mathcal{D}$; $\lambda$ is a pole of $(H -\lambda^2)^{-1}$, if and only if $\lambda^2 < 0$ is an eigenvalue of $H$. Furthermore, for $\chi \in C_0^\infty(\mathbb R; [0,1])$ with $\chi =1 $ near $[-R_0, R_0]$, the cutoff resolvent $\chi (H - \lambda^2)^{-1} \chi$ continues meromorphically $L^2(\R) \to \mathcal{D}$ from $\imag \lambda \gg 1$ to $\C$. The poles of the continuation are known as its \textit{resonances}. 

\begin{proof}[Proof of Theorem \ref{unif resolv est thm}]
Throughout, we use $C$ to denote a positive constant, which may depend on the operator coefficients in \eqref{H} and on $\lambda_0$, and whose precise value may change from line to line, but is always independent of $\lambda$. The quantity $\theta_0$ in the statement of the Theorem \ref{unif resolv est thm}, which restricts $|\imag \lambda|$, will be fixed sufficiently small (depending on $\lambda_0$) at the appropriate step.

We first establish \eqref{unif resolv est} for $k = 0$, $\imag \lambda > 0$, and $|\real \lambda | \ge \lambda_0$.  In this case, let us expand
\begin{equation} \label{expand lambda square}
H - \lambda^2 =  H  - (\real \lambda)^{2} + (\imag \lambda)^2 - i 2\real \lambda \cdot \imag \lambda.  
\end{equation}
If $\imag \lambda \ge \theta_0$, then by the spectral theorem for self-adjoint operators,
\begin{equation} \label{use spect thm}
\|\chi (H - \lambda^2)^{-1} \chi\|_{L^2 \to L^2} \le  \frac{1}{2 |\real \lambda| \imag \lambda }  \le \frac{1}{2 \theta_0 |\real \lambda|}  \le C |\real \lambda|^{-1}.
\end{equation}

If $0 < \imag \lambda < \theta_0$, we rescale \eqref{expand lambda square} semiclassically:

\begin{gather}
h = |\real \lambda|^{-1}, \quad  h_0 = \lambda_0^{-1},  \quad \ep_0 = 2h_0 \theta_0, \quad E(h) = 1 - h^2(\imag \lambda)^2, \quad \ep(h) = 2h \sgn(\real \lambda)  \imag \lambda, \label{semiclassical rescale} \\
 \tilde{V} = h^2 V, \quad \tilde{b} = hb. \label{rescale V and b}
\end{gather}
  Then
\begin{equation*} 
\begin{split}
&(H-\lambda^2)v \\
 &= \beta d(-\alpha v' - ib v) + \beta b D_x v + vV - ((\real \lambda)^2 + 2i \real \lambda \imag \lambda - (\imag \lambda)^2)v \\
& = h^{-2} (\beta d(- h^2 \alpha v' -  h^2 ib v) +  h^2 \beta b D_x v +   vh^2 V -  (1 - h^2 (\imag \lambda)^2)v - 2i h  \imag \lambda v) \\
& = h^{-2} (\tilde{H} - E(h) - i \ep(h))v
\end{split}
\end{equation*}
where
\begin{equation} \label{tilde H}
\tilde{H}v = \beta d(-h^2 \alpha v' -i h \tilde{b} v) + h \beta \tilde{b} D_x v + v \tilde{V}.
\end{equation}

We check how, in the context of \eqref{semiclassical rescale}, \eqref{rescale V and b}, and \eqref{tilde H}, the hypotheses of Theorem \ref{Carleman mag schro 1D thm} are satisfied. Let us now fix $\theta_0 = 2^{-1/2}h^{-1}_0$, so that, by \eqref{semiclassical rescale}, $1/2 \le E(h) \le 1$. Since $V_1$, $b_1$, $1 - \alpha$, and $1 - \beta$ are supported in $[-R_0, R_0]$, with respect to the quantity \eqref{general inf},
\begin{equation} \label{full range of h}
\inf_{|x| \ge 2R_0} (\beta^{-1} \alpha (E(h)- h^2V_1) + h^2b^2_1) \ge \frac{1}{2}, \qquad h \in (0, h_0].
\end{equation}
Furthermore, since $|h^2V_1| \le h^2 \| V_1\|_{L^\infty} \to 0$ as $h \to 0^+$, there exists $0 < \tilde{h}_0 < h_0$ (depending on $V_1$, $\alpha$, and $\beta$) so that 
\begin{equation} \label{smaller range of h}
\inf_{\R} (\beta^{-1} \alpha (E(h)- h^2V_1) + h^2b^2_1) \ge \inf \frac{\alpha}{4\beta}, \qquad h \in (0, \tilde{h}_0].
\end{equation}

Thus the assumptions of Theorem \ref{Carleman mag schro 1D thm} hold, and we have a weighted estimate as in \eqref{pre penult est}. It is now important to  consider the properties of $\varphi(\cdot, h)$ and $C_w$ in \eqref{pre penult est}. Since $E(h)$, $|h^2V_1|$, and $h^2b^2_1$ are uniformly bounded for $h \in (0,h_0]$, we may take $\varphi$ independent of $h$ for $h \in  (0, h_0]$. Moreover, since since the infima in \eqref{smaller range of h} are over $\R$, we can set $\varphi \equiv 0$ for $h \in (0, \tilde{h}_0]$. As for $C_w$, we look to \eqref{Cw}, \eqref{weta}, \eqref{mu}, and \eqref{fix Wj}. From these it follows that $C_w$ is uniformly bounded for $h \in (0, h_0]$, because by \eqref{rescale V and b} the measure $\mu$ in \eqref{mu} obeys
\begin{equation*}
\begin{split}
\|\mu\|(\R) &\le C \big( \| \alpha^{-1} (b_1^2 - b^2)  + \beta^{-1} V_0\| + \| \alpha^A d(\varphi')| \| + \| (\varphi')^A d\alpha \|  + \||\varphi'|(b^2 + |b|)\|  \\ 
&+\| d(\beta^{-1} \alpha(E(h) - V_1) + \alpha^2 (\varphi')^2 + b_1^2)\| \big)(\R),
\end{split}
\end{equation*}
the right side of which is uniformly bounded for $h \in (0,h_0]$. We conclude that \eqref{pre penult est} simplifies to
\begin{equation*}
\begin{split}
\int& \langle x \rangle^{-2s}  (|u|^2 +  |h \alpha u' + i\tilde{b}u|^2)  dx \\
&\le C h^{-2} \int |f|^2 dx +  C  h^{-1} |\ep(h)| \int  |h \alpha u' + i\tilde{b}u|^2 +  |u|^2 dx, \qquad h \in (0,h_0], \, \imag \lambda \in (0, \theta_0]
\end{split}
\end{equation*}
where $v= (\tilde{H} - E(h) - i \ep(h))^{-1} \langle x \rangle^{-s} f$ for $f \in L^2(\R)$, and $u = e^{\varphi/h}v$.

Follow again the steps from \eqref{h alpha u prime plus b square} to \eqref{ult est}; each constant that appears and involves $E(h)$, $\tilde{V}$, $\tilde{b}$, $\alpha$, $\beta$, or $\varphi$ is bounded above uniformly in $h \in (0, h_0]$. Therefore the version of \eqref{ult est} we arrive at is 
\begin{equation*}
\int \langle x \rangle^{-2s} ( |u|^2 + |h \alpha u' + i\tilde{b}u|^2) dx \le  Ch^{-2}  \int |f|^2 dx + C  h^{-1} |\ep(h)| \int  |v|^2 dx.
\end{equation*}

The term $|\ep| \int |v|^2 dx$ can be bounded with an argument similar to  \eqref{ep int u square exp} to get
\begin{equation*} 
\begin{split}
\int \langle x \rangle^{-2s}  |v|^2 dx \le   C h^{-2} \int |f|^2 dx.
\end{split}
\end{equation*}
which implies,
\begin{equation} \label{H tilde est}
\| \langle x \rangle^{-s}  (\tilde{H} - E(h) - i \varepsilon(h))^{-1}  \langle x \rangle^{-s}\|   \le Ch^{-1}, \qquad h \in (0,h_0], \, \imag \lambda \in (0, \theta_0].
\end{equation} 
 Replacing $\langle x \rangle^{-s}$ by $\chi$ in the left side of \eqref{H tilde est} does not affect the right side. Using this with $h = |\real \lambda|^{-1}$ and $ (\real \lambda)^{2} (H - \lambda^2)^{-1} = (\tilde{H} - E(h) - i \varepsilon(h))^{-1}$ yields
 
 \begin{equation} \label{use semiclassical rescale} 
\|\chi (H - \lambda^2)^{-1} \chi\|_{L^2 \to L^2} \le C |\real \lambda|^{-1}, \qquad |\real \lambda | \ge \lambda_0, \, 0 <  \imag \lambda \le \theta_0.
\end{equation}

Next, we adapt the proof of \cite[Proposition 2.5]{bgt04} to show
\begin{equation} \label{L2 H1 bd uhp}
\|\chi (H - \lambda^2)^{-1} \chi\|_{L^2 \to H^1} \le C, \qquad |\real \lambda | \ge \lambda_0, \, 0 < \imag \lambda \le 1.
\end{equation}
We employ the notation, 
\begin{equation} \label{apply op}
(H - \lambda^2) v = \chi f, \qquad |\real \lambda | \ge \lambda_0, \, 0 < \imag \lambda \le 1, \, f \in L^2(\R), \, v \in \mathcal{D}, 
\end{equation}
and make use of additional cutoffs
\begin{equation} \label{cutoffs}
 \chi_1,\, \chi_2 \in C^\infty_0(\R ; [0,1]), \quad \chi_1 = 1 \text{ on } \supp \chi, \quad \chi_2 = 1 \text{ on } \supp \chi_1. 
\end{equation}
Observe
\begin{equation*}
\| \chi(H - \lambda^2)^{-1} \chi f \|_{H^1} \le \| \chi v \|_{L^2} + \|(\chi v)'\|_{L^2} \le C( \| \chi_2 v \|_{L^2} + \|\chi_1 v'\|_{L^2}),
\end{equation*}
Since we already have \eqref{use spect thm} and \eqref{use semiclassical rescale}, it suffices to show
\begin{equation} \label{chi1 u prime}
\|\chi_1 v'\|^2_{L^2} \le C \big( (|\real \lambda| +1)^2 \|\chi_2 v\|_{L^2}^2 + \|\chi_2 f\|_{L^2}^2 \big).
\end{equation}

Multiplying \eqref{apply op} by $\chi^2_1 \overline{v}$ and applying \eqref{prod to form} gives
\begin{equation*}
\begin{split}
\int & \chi_1^2 \chi f \overline{v} \beta^{-1} dx \\
&= \langle \chi^2_1 v, Hv \rangle_{L^2(\beta^{-1}dx)} - \lambda^2 \int \chi^2_1 |v|^2 \beta^{-1} dx \\
&= \int \alpha (\chi_1^2 \overline{v})' v dx + i \int b( (\chi_1^2 \overline{v})' v -  \chi_1^2 \overline{v} v')dx  + \int \beta^{-1} \chi^2_1 |v|^2 V - \lambda^2 \int \chi^2_1 |v|^2 \beta^{-1} dx     \\
&= \int \alpha \chi_1^2 |v'|^2dx + 2 \int \alpha \chi_1' \overline{v} \chi_1 v' +i  b \chi_1 \chi'_1 |v|^2dx - \lambda^2 \int \chi^2_1 |v|^2 \beta^{-1} dx\\
&+ i\int b\chi^2_1(\overline{v}'v - \overline{v}v')dx + \int \beta^{-1} \chi^2_1 |v|^2 V 
\end{split}
\end{equation*}
Taking the real part of both sides, and estimating the last line in a manner similar to \eqref{bdd below}, we find
\begin{equation*} 
\begin{split}
\int \chi_1^2 |v'|^2dx \le C\big( \|\chi_2 f\|^2_{L^2}  + (|\real \lambda| + 1)^2 \|\chi_2 v\|^2_{L^2} \big) + \frac{1}{2} \int \chi_1^2 |v'|^2dx. 
\end{split}
\end{equation*}
Absorbing the last term on the right side into the left side confirms \eqref{chi1 u prime}.

By \eqref{use spect thm}, \eqref{use semiclassical rescale}, and  \eqref{L2 H1 bd uhp}, for $|\real \lambda | \ge \lambda_0, \, 0 < \imag \lambda \le 1$, and $f \in L^2(\R)$,
\begin{equation*}
\begin{split}
\| \chi (H- \lambda^2)^{-1} \chi f \|_{\mathcal{D}} & \le   \| \chi (H- \lambda^2)^{-1} \chi f \|_{L^2} + \| H\chi (H- \lambda^2)^{-1} \chi f \|_{L^2} \\
&\le  \| \chi (H- \lambda^2)^{-1} \chi f \|_{L^2} + \| [-\partial^2_x, \chi] \chi_1 (H- \lambda^2)^{-1} \chi f \|_{L^2} \\
& + \| \chi ((H - \lambda^2) + \lambda^2) (H- \lambda^2)^{-1} \chi f \|_{L^2} \\
&\le  \| \chi (H- \lambda^2)^{-1} \chi f \|_{L^2} +  \| f \|_{L^2}\\
&+ \| [-\partial^2_x, \chi] \chi_1 (H- \lambda^2)^{-1} \chi f \|_{L^2} + (|\real \lambda| + 1)^2 \| \chi (H- \lambda^2)^{-1} \chi f \|_{L^2} \\
&\le C (|\real \lambda| + 1) \| f \|_{L^2}.
\end{split}
\end{equation*}
This implies \eqref{unif resolv est D} for  $|\real \lambda | \ge \lambda_0, \, 0 < \imag \lambda \le 1$, and that the continued resolvent $L^2(\R) \to \mathcal{D}$ has no poles in $\R \setminus \{0\}$ (since $\lambda_0 > 0$ is arbitrary).

The last operator norm bound we prove is 
\begin{equation} \label{H1 L2 bf uhp}
\| \chi (H - \lambda^2)^{-1} \chi \|_{H^1 \to L^2} \le C |\real \lambda|^{-1}, \qquad | \real \lambda | \ge \lambda_0, \, 0 < \imag \lambda \le 1. 
\end{equation}
For this, we employ the same notation as in \eqref{apply op} and \eqref{cutoffs}, except now we suppose $f \in H^1(\R)$. 

From the proof of Lemma \ref{self-adjointness 1D lemma}, the form domain of $(H ,\mathcal{D})$ is $H^1(\R)$, so there exists a sequence $f_k \in \mathcal{D}$ converging to $f$ in $H^1(\R)$, and corresponding functions $v_k \defeq (H - \lambda^2)^{-1} \chi f_k$ converging to $v$ in $(\mathcal{D}, \| \cdot \|_{\mathcal{D}})$. Since $Hv_k = (H -\lambda^2)^{-1} \chi_1 H \chi f_k$,
\begin{equation} \label{chi H u}
\begin{split}
\| \chi H v \|_{L^2} = \lim_{k \to \infty} \| \chi H v_k \|_{L^2} \le \lim_{k \to \infty} \|  \chi_1  (H -\lambda^2)^{-1} \chi_1 H \chi f_k\|_{L^2}.
\end{split}
\end{equation}
Furthermore, by \eqref{prod to form}, for any $g \in L^2(\R)$, 
\begin{equation*}
\begin{split} 
|\langle \chi_1  (H -\lambda^2)^{-1} \chi_1 H \chi f_k, g \rangle_{L^2}| &= |\langle \chi_1  (H -\lambda^2)^{-1} \chi_1 H \chi f_k, \beta g \rangle_{L^2(\beta^{-1}dx)}| \\
&= |\langle  H \chi f_k, \chi_1  (H -(-\overline{\lambda})^2)^{-1} \chi_1 \beta g \rangle_{L^2}| \\
&\le C \| \chi f_k \|_{H^1} \| \chi_1  (H -(-\overline{\lambda})^2)^{-1} \chi_1 \beta g\|_{H^1}. 
\end{split} 
\end{equation*}
Because $\| \chi_1  (H -(-\overline{\lambda})^2)^{-1} \chi_1 \beta v\|_{H^1} \le C \| v\|_{L^2}$ by \eqref{L2 H1 bd uhp},
\begin{equation*}
 \|\chi_1  (H -\lambda^2)^{-1} \chi_1 H \chi f_k\|_{L^2} \le C \| \chi f_k \|_{H^1}. 
\end{equation*}
Returning to \eqref{chi H u}, we now find 
\begin{equation*} 
\| \chi H u \|_{L^2} \le C \lim_{k \to \infty}  \| \chi f_k \|_{H^1} \le C \| f \|_{H^1}, \qquad |\real \lambda| \ge \lambda_0, \, 0 < \imag \lambda \le 1.
\end{equation*}

Thus we have established \eqref{unif resolv est} and \eqref{unif resolv est D} in the upper half plane. To show these estimates continue to hold in a strip in the lower half plane, we appeal to a standard resolvent identity argument due to Vodev \cite[Theorem 1.5]{vo14}. In fact, the corresponding steps from \cite[Section 6]{lash24} can be followed with no changes.\\
\end{proof}

\subsection{Consequences for the Schr\"odinger and wave propagators}
With Theorem \ref{unif resolv est thm} in hand, we prove Corollary \ref{schro wave corollary}. The strategies we employ to conclude \eqref{time integrability schro} and \eqref{LED cosine wave}, and \eqref{LED sine wave} are well-known, see \cite[Sections 2.3 and 7.1]{dz}, and are based on Stone's formula.  

For the meromorphic continuation of the operator $H$, we utilize the notation
\begin{equation*}
R(\lambda) \defeq \chi(H - \lambda^2)^{-1} \chi. 
\end{equation*}

\begin{proof}[Proof of \eqref{time integrability schro}]
Fix $T > 0$ and $\varphi \in C^\infty_0(0, \infty)$. Define the operator $A : L^2(\R) \to L^2((-T, T)_t \times \R_x)$ by $v \mapsto  \chi  \varphi(H) e^{-itH}v$. We show there exists $C_1$ independent of $T$ and $\varphi$ so that $\| AA^* f \|^2 \le C^2_1 \|f\|^2_{L^2(\R_t \times \R_x)}$ for all $f \in C^\infty_0((-T,T)_t \times \R_x)$. Then by $\| A \|^2 = \|AA^*\|$ and the density of $C^\infty_0((-T,T)_t \times \R_x)$ in $L^2((-T, T)_t \times \R_x)$,
\begin{equation*}
\int_{\R} \mathbf{1}_{[-T, T]} (t) \| \chi e^{-iHt} \varphi(H) v \|_{L^2}^2 dt  = \int_{-T}^{T} \| \chi e^{-iHt} \varphi(H) v \|_{L^2}^2 dt \le C_1 \| v\|^2_{L^2} , \qquad v \in L^2(\R). 
\end{equation*}
Because $C_1$ is independent of $T$ and $\varphi$, we conclude \eqref{time integrability schro} by applying the monotone convergence theorem twice: first to a sequence of $\varphi$'s increasing up to the indicator function of $[0, \infty)$, and then to sequence of $T$'s tending to infinity.
 
A straightforward calculation demonstrates
\begin{equation*}
A^*f = \int_{\R} e^{isH} \varphi(H) \chi f(s, \cdot) ds, \qquad f \in C^\infty_0((-T,T)_t \times \R_x),
\end{equation*}
and thus
\begin{equation*}
AA^*f = \int_{\R} \chi e^{-i(t -s) H}\varphi^2(H) \chi f(s, \cdot) ds.
\end{equation*}
Now use Stone's formula \cite[Section 4.1]{te} to expand $\chi e^{-i(t -s) H}\varphi^2(H)\chi f(s, \cdot) $ in the sense of strong convergence in $L^2(\R)$: 
\begin{equation} \label{Stone's formula schro propagator}
\begin{split}
\chi e^{-i(t -s) H}\varphi^2(H)\chi &= \lim_{\ep \to 0^+} \frac{1}{2 \pi i} \int_0^\infty e^{-i(t-s) \tau} \varphi^2(\tau) \chi \big[ (H - (\tau + i\ep))^{-1} - (H - (\tau - i \ep)^{-1}) \big] \chi  d\tau \\
&=  \frac{1}{2 \pi i} \int_0^\infty e^{-i(t-s) \tau} \varphi^2(\tau) [ R(\sqrt{\tau}) - R(-\sqrt{\tau}) \big] d\tau.
\end{split}
\end{equation}
We are able to set $\ep = 0$ due to $R(\lambda)$ having the meromorphic continuation supplied by Theorem \ref{unif resolv est thm}. We now have, by Fubini's theorem,
\begin{equation*}
\begin{split}
AA^*f &= \frac{1}{2 \pi i} \int_{\R} \int_0^\infty e^{-i(t-s) \tau} \varphi^2(\tau) [ R(\sqrt{\tau}) - R(-\sqrt{\tau}) \big]  f(s, \cdot) d\tau ds\\
 &=  \frac{1}{2 \pi i} \int^\infty_0 e^{-it \tau} \big[ R(\sqrt{\tau}) - R(-\sqrt{\tau}) \big] \varphi^2(\tau) \int_{\R} e^{is\tau} f(s, \cdot) ds d\tau \\
 &= -i \mathcal{F}_{\tau \mapsto t} \big[(R(\sqrt{\tau}) - R(-\sqrt{\tau})) \varphi^2(\tau) \mathcal{F}^{-1}_{s \mapsto \tau} (f(s, \cdot))  \big],  
\end{split}
\end{equation*}
where $\mathcal{F}$ and $\mathcal{F}^{-1}$ denote Fourier transform and inverse Fourier transform respectively. Applying Plancherel's theorem twice:
\begin{equation*}
\begin{split}
\|AA^*f\|^2 &=  4\pi^2  \| (R(\sqrt{\tau}) - R(-\sqrt{\tau})) \varphi^2(\tau) \mathcal{F}^{-1}_{s \mapsto \tau} (f(s, \cdot)) \|^2_{L^2(\R \times \R)} \\
& \le 4 \pi^2  \sup_{\tau \ge 0}(\| (R(\sqrt{\tau}) - R(-\sqrt{\tau})) \varphi^2(\tau) \|^2_{L^2 \to L^2})\|f \|^2_{L^2(\R \times \R)}. 
 \end{split}
\end{equation*}
Clearly $\sup_{\tau \ge 0} \| (R(\sqrt{\tau}) - R(-\sqrt{\tau})) \varphi^2(\tau) \|_{L^2 \to L^2}$ is independent of the support of $f$. In addition, \eqref{unif resolv est} and the hypothesis that $H$ has no zero resonance imply $\| (R(\sqrt{\tau}) - R(-\sqrt{\tau})) \varphi^2(\tau) \|^2_{L^2 \to L^2}$ is bounded independent of $\varphi$ too. This completes the proof of \eqref{time integrability schro}. \\
\end{proof}

\begin{proof}[Proof of \eqref{LED cosine wave}]
Let
\begin{equation} \label{X}
X(t) = e^{C_3t} 
\end{equation}
for $C_3 > 0$ to be chosen. First, decompose the wave propagator according to $X(t)$,
\begin{equation} \label{decompose wave propagator}
\chi \cos(t \sqrt{|H|})\mathbf{1}_{\ge 0}(H) \chi v =  \chi \cos(t \sqrt{|H|})\mathbf{1}_{[0, X(t)]}(H) \chi v + \chi \cos(t \sqrt{|H|}) \mathbf{1}_{ \ge X(t)}(H) \chi v.
\end{equation}
Estimate the second term on the right side of \eqref{decompose wave propagator} using the spectral theorem,
\begin{equation} \label{spectral thm cosine propagator}
\begin{split}
\|\cos(t \sqrt{|H|}) & \mathbf{1}_{ \ge X(t)}(H) \chi v\|_{L^2} \\
&\le \big\| \frac{ \cos(t \sqrt{|H|})}{\sqrt{|H|}} \mathbf{1}_{ \ge X(t)}(H) \big\|_{L^2 \to L^2} \| \sqrt{|H|} \chi v\|_{L^2} \le e^{-C_3 t/ 2} \| \sqrt{|H|} \chi v\|_{L^2}. 
\end{split}
\end{equation}

For the first term on the right side of \eqref{decompose wave propagator}, we use Stone's formula and the change of variable $\tau \mapsto \lambda^2$,
\begin{equation} \label{Stone's formula cosine propagator}
\begin{split}
\chi\cos(t \sqrt{|H|})&\mathbf{1}_{[0, X(t)]}(H) \chi v \\
&= \lim_{\ep \to 0^+} \frac{1}{2\pi i} \int_0^{X(t)} \cos(t \sqrt{\tau}) \chi \big[ (H - (\tau + i\ep))^{-1} - (H - (\tau - i \ep))^{-1} \big] \chi v d\tau \\
&\lim_{\ep \to 0^+} \frac{1}{2\pi i} \int_0^{X(t)}  \lambda (e^{it \lambda } + e^{-it \lambda })   \chi \big[ (H - (\lambda^2 + i\ep))^{-1} - (H - (\lambda^2 - i \ep))^{-1} \big] \chi v d\tau \\
&= \frac{1}{2\pi i} \int_{-X(t)}^{X(t)} \lambda e^{-it \lambda } [ R(\lambda) - R(-\lambda) \big] v d\lambda.
\end{split}
\end{equation}
We may set $\ep  = 0$ in line three due to Theorem \ref{unif resolv est thm} and our hypothesis that $R(\lambda)$ has no resonance at zero. In particular, by \eqref{unif resolv est}, there exists $\theta_1 > 0$ sufficiently small so that $R(\lambda)$ is analytic near the strip $-\theta_1  \le \imag \lambda \le 0$, and
\begin{equation} \label{resolv est in strip for wave decay}
\| \lambda (R(\lambda) - R(-\lambda)) \|_{H^k \to L^2} \le C(|\real \lambda| +1)^{-k}, \qquad -\theta_1  < \imag \lambda < 0, \, k \in \{0,1\},
\end{equation}
for some $C > 0$ independent of $|\real \lambda|$. 

Deform the contour in the last line of \eqref{Stone's formula cosine propagator} into the lower half plane,
\begin{equation} \label{deform contour cosine propagator}
\begin{split}
 \int_{-X(t)}^{X(t)} \lambda e^{-it \lambda }&[ R(\lambda) - R(-\lambda) \big] v d\lambda  \\
 &=   e^{-\theta_1 t} \int_{\real \lambda =  -X(t)}^{\real \lambda = X(t)} (\lambda e^{-it \real \lambda } [ R(\lambda) - R(-\lambda) \big]) \rvert_{\imag \lambda = -\theta_1}v d( \real \lambda) \\
&+ \int_{\imag \lambda = -\theta_1 }^{\imag \lambda = 0} ( \lambda e^{-it \lambda } [ R(\lambda) - R(-\lambda) \big] ) \rvert_{\real \lambda = X(t)} v d (\imag \lambda) \\
&- \int_{\imag \lambda = -\theta_1 }^{\imag \lambda = 0} ( \lambda e^{-it  \lambda } [ R(\lambda) - R(-\lambda) \big] ) \rvert_{\real \lambda = -X(t)} v d (\imag \lambda)
\end{split}
\end{equation} 
Denote the terms in lines two, three, and four of \eqref{deform contour cosine propagator} by $I_2$, $I_3$, and $I_4$, respectively. For some $C > 0$ independent of $t$, and $v$,
\begin{equation} \label{est contour pieces cosine}
\begin{gathered}
\| I_2 \|_{L^2} \le C X (t) e^{-\theta_1 t} \|v\|_{H^1} =  C e^{(C_3-\theta_1 )t} \|v\|_{H^1} , \\
\| I_3 \|_{L^2}, \| I_4 \|_{L^2} \le C X^{-1} (t) \|v\|_{H^1}= C e^{-C_3t} \|v\|_{H^1}.
\end{gathered}
\end{equation}
Setting now $C_3 = \theta_1/2$, \eqref{spectral thm cosine propagator} and \eqref{est contour pieces cosine} conclude the proof of \eqref{LED cosine wave}.\\
\end{proof}

\begin{proof}[Proof of \eqref{LED sine wave}]
The proof of \eqref{LED sine wave} is similar to the proof of \eqref{LED cosine wave}. We use the same $X(t)$ as in \eqref{X}, and this time find,
\begin{equation*} 
\big \|\frac{\sin(t \sqrt{|H|})}{\sqrt{|H|}} \mathbf{1}_{ \ge X(t)}(H) \chi v\|_{L^2}  \le e^{-C_3 t/ 2} \| v\|_{L^2},
\end{equation*}
and 
\begin{equation*} 
\chi\frac{\sin(t \sqrt{|H|})}{\sqrt{|H|}} \mathbf{1}_{[0, X(t)]}(H) \chi v = \frac{1}{2\pi } \int_{-X(t)}^{X(t)} e^{-it \lambda } [ R(\lambda) - R(-\lambda) \big] v d\lambda.
\end{equation*}
Once again, we deform the contour as in \eqref{deform contour cosine propagator}, apply \eqref{resolv est in strip for wave decay}, and fix $C_3 = \theta_1/2$. This establishes \eqref{LED sine wave}.\\
\end{proof}
 
 \appendix

\section{Self-adjointness of the magnetic Schrödinger operator} \label{self adj oned section}
 
 In this appendix, we show
 \begin{equation*} 
 P(h) = \beta ( -h^2\partial_x(\alpha \partial_x) + hbD_x + hD_xb) +V : L^2(\R ; \beta^{-1} dx) \to L^2(\R; \beta^{-1}dx ),
 \end{equation*} 
is self-adjoint with respect the domain $\mathcal{D}$ defined in \eqref{D}. The assumptions on the coefficients are more general than those prescribed in Section \ref{intro section}. Suppose
 \begin{gather}
 \alpha, \, \beta \in L^\infty(\R; (0, \infty)) \text{ and $\inf \alpha, \, \inf \beta > 0$, } \label{alpha beta appendix} \\
 V = V_0 + V_1, \text{ where $V_0$ is a finite signed Borel measure on $\R$ and $V_1 \in L^\infty(\R; \R)$}, \label{V appendix}\\
 b = b_0 + b_1, \text{ where $b_0 \in L^2(\R; \R)$ and $b_1 \in  L^\infty(\R; \R)$.} \label{b appendix} 
 \end{gather} 
Since $\beta$ is bounded from above and below by positive constants, $L^2(\R; \beta^{-1} dx)  = L^2(\R;  dx)$ and their norms are equivalent.

 \begin{lemma} \label{self-adjointness 1D lemma}
Under \eqref{alpha beta appendix}, \eqref{V appendix}, and \eqref{b appendix}, the subspace $\mathcal{D}$ given by \eqref{D} is dense in $L^2(\beta^{-1}dx)$. The operator $P(h)$ equipped with domain $\mathcal{D}$ and defined by \eqref{P(h)u} and is self-adjoint on $L^2(\R ; \beta^{-1}dx)$.
 \end{lemma} 

  \begin{proof}
Since multiplication by $V_1$ is a bounded operator on $L^2(\R; \beta^{-1} dx)$, by the Kato-Rellich theorem \cite[Theorem 6.4]{te}, we suppose $V_1 = 0$ without loss of generality. 

Let $\mathcal D_\textrm{max} \supseteq \mathcal{D}$ be the set of $u \in L^2(\R)$ such that $-h \alpha u' - ibu$ has locally bounded variation, and $P(h)u$, defined in the distributional sense by \eqref{P(h)u}, belongs to $L^2(\R)$. We use $u_c$ to denote the unique absolutely continuous representative of $u \in \mathcal{D}_\textrm{max}$. 

First we prove $\mathcal D_\textrm{max} \subseteq \mathcal D$. Since the reverse containment is trivial, we will conclude $\mathcal D_\textrm{max} = \mathcal D$. Our strategy is as follows. For $u \in \mathcal D_{\textrm{max}}$, fix a representative $f$ of $-h^2 \alpha u' - i h b u$ with locally bounded variation. If necessary, modify $f$ on a set of Lebesgue measure zero so $f^A(x) = f(x)$ for all $x \in \R$ (this simplifies steps that involve \eqref{ftc} or \eqref{e:prod}).  For $a > 0$, define
\begin{equation*}
 \mathsf{x} \defeq \| f\|_{L^2(-a,a)}, \qquad  \mathsf{y} \defeq \sup_{[-a,a]} |u_c|, \qquad \mathsf{z} \defeq  \sup_{[-a,a]} |f|.
\end{equation*}
We show that a system of inequalities holds: 
\begin{gather}
\mathsf{x}^2 \le  C_1 + C_2 \mathsf{x} + C_3\mathsf{z} + C_4\mathsf{y}^2 + C_5\mathsf{y} \mathsf{z},  \label{bd x square} \\
 \mathsf{y}^2 \le C_6 + C_7\mathsf{x},  \label{bd y square} \\  
 \mathsf{z}^2 \le C_8 + C_{9}\mathsf{x} + C_{10} \mathsf{y} \mathsf{z}.  \label{bd z square}
\end{gather}
for constants $C_j > 0$, $1 \le j \le 10$, which may depend on $h$, but are independent of $a$. After using \eqref{bd y square} to eliminate $\mathsf{y}$ from \eqref{bd x square} and \eqref{bd z square} , we obtain a system in $\mathsf{x}$ and $\mathsf{z}$ with quadratic left sides and subquadratic right sides. Hence $\mathsf{x}$, $\mathsf{y}$, and $\mathsf{z}$ are bounded in terms of the $C_j$. In particular $u, f \in L^\infty(\R)$, $f \in L^2(\R)$, and since $u' = -h^{-2} \alpha^{-1} f - i h^{-1} \alpha^{-1} bu$ Lebesgue almost everywhere, $u' \in L^2(\R)$ too. Thus  $\mathcal D_{\max} \subseteq \mathcal D$.

We now turn to establishing \eqref{bd x square}, \eqref{bd y square}, and \eqref{bd z square}. For $u \in \mathcal D_\textrm{max}$, $f = -h^2\alpha u_c' - ihbu_c$ Lebesgue almost everywhere, so
\begin{equation} \label{bound L2 u prime 1}
\tfrac{1}{\sup \alpha} \int_{(-a,a)} |f|^2dx \le \int_{(-a,a)} \tfrac{1}{\alpha} |f|^2dx = \int_{(-a,a)}  (-h^2 \overline{u}_c' + i h \tfrac{b}{\alpha}\overline{u}_c) f dx. 
\end{equation}
By \eqref{e:prod} and $f^A = f$, $\overline{u}'_c f = d(\overline{u}_c f ) - \overline{u}_c df$. Inserting this into the right side of \eqref{bound L2 u prime 1} and applying \eqref{ftc} gives
\begin{equation} \label{bound L2 u prime 2}
\begin{split}
\tfrac{1}{\sup \alpha}  \int_{(-a,a)} |f|^2dx &\le  \tfrac{h}{\inf \alpha} (\| b_0\|_{L^2} \| u \|_{L^2} \mathsf{z}  + \|b_1\|_{L^\infty}  \|u\|_{L^2}  \mathsf{x}) \\
&-h^{2}((\overline{u}_cf)^{L}(a) -(\overline{u}_cf)^{R}(-a)) + h^{2} \int_{(-a,a)}  \overline{u}_c df  \\
&\le \tfrac{h}{\inf \alpha} (\| b_0\|_{L^2} \| u \|_{L^2} \mathsf{z}  + \|b_1\|_{L^\infty}  \|u\|_{L^2}  \mathsf{x}) + 2h^{2} \mathsf{y} \mathsf{z} +h^{2} \int_{(-a,a)} \overline{u}_c df.  
\end{split}
\end{equation}
Now use that, as measures on bounded Borel subsets of $\R$,
\begin{equation*}
\begin{split}
 df &= \beta^{-1} P(h)u - h b D_x u - \beta^{-1} u_c V_0  \\
&= \beta^{-1}P(h)u -i \tfrac{b}{h \alpha}f + \tfrac{b^2}{\alpha} u_c - \beta^{-1} u_c V_0.
\end{split}
\end{equation*}
Thus, since $b^2 \le 2(b^2_0 + b^2_1)$,
\begin{equation} \label{int uc df}
\begin{split}
 \int_{(-a,a)} \overline{u}_c df &= \int_{(-a,a)}  \overline{u}_c (\beta^{-1}P(h)u -i \tfrac{b}{h \alpha}f + \tfrac{b^2}{\alpha} u_c)dx  - \int_{(-a,a)}  \beta^{-1} |u_c|^2 V_0  \\
 &\le \tfrac{1}{\inf \beta}  \|P(h)u\|_{L^2}\|u\|_{L^2} + \tfrac{\|b_0\|_{L^2} \| u\|_{L^2}}{h \inf \alpha} \mathsf{z} +\tfrac{\|b_1\|_{L^\infty} \| u\|_{L^2}}{h \inf \alpha} \mathsf{x} \\
&+ \tfrac{2 \|b_0\|^2_{L^2}}{\inf \alpha} \mathsf{y}^2 + \tfrac{2 }{\inf \alpha} \|b_1\|^2_{L^\infty}  \|u\|^2_{L^2} +\tfrac{
\|V_0\|}{\inf \beta}  \mathsf{y}^2.
\end{split}
\end{equation}
Combining \eqref{bound L2 u prime 1}, \eqref{bound L2 u prime 2} and \eqref{int uc df} yields a bound of the form \eqref{bd x square}.

Next,
\begin{equation*}
\begin{split}
\sup_{[-a,a]}|u_c|^2 &=\sup_{x \in [-a,a]}  \big(|u_c(0)|^2 + 2 \real \int_0^x \overline{u}'_c u_c  dx \big) \\
&=  \sup_{x \in [-a,a]}  \big(|u_c(0)|^2 + 2 \real \int_0^x (- \tfrac{1}{h^2 \alpha} \overline{f} + \tfrac{ib}{h \alpha}\overline{u}_c) u_c  dx \big) \\
 &\le |u_c(0)|^2 + \tfrac{2 \|u\|_{L^2}}{h^2 \inf \alpha} \mathsf{x},
 \end{split}
\end{equation*}
which is \eqref{bd y square}. 

If $x \in (0, a)$ and $f$ is continuous at $x$ then by \eqref{ftc}, \eqref{e:prod} and $df = \beta^{-1} P(h)u - h b D_xu - \beta^{-1} u_c V_0$,
\begin{equation} \label{bound u prime}
\begin{split}
|f&|^2(x) \\
&= (|f|^2)^{R}(0) + \int_{(0,x]} d (\overline{f} f)  \\
&=  (|f|^2)^{R}(0)+ 2 \real \int \overline{f} df \\
&= (|f|^2)^{R}(0) + 2 \real \Big( \int_{(0, x]}\overline{f} (\beta^{-1} P(h)u - hbD_xu)d x -  \int_{(0, x]} \beta^{-1} \overline{f} u_c V_0 \Big) \\
&\le (|f|^2)^{R}(0) + \tfrac{2 \|V_0\|}{\inf \beta} \mathsf{y} \mathsf{z} \\
&+  2 \real \int_{(0, x]}\overline{f} (\beta^{-1} P(h)u - hbD_xu)d x.
\end{split}
\end{equation}
Computing further, using $-hbD_xu = -ih^{-1}\alpha^{-1}b f + \alpha^{-1} b^2 u_c$ Lebesgue almost everywhere, and again $b^2 \le 2b^2_0 + 2b^2_1,$
\begin{equation} \label{bound u prime 2}
\begin{split}
2 \real \int_{(0, x]}&\overline{f} (\beta^{-1} P(h)u - hbD_xu)d x = 2 \real \int_{(0, x]} \overline{f} (\beta^{-1} P(h)u -i \tfrac{b}{h \alpha}f + \tfrac{b^2}{\alpha} u_c )d x\\
& \le \tfrac{2}{\inf \beta} \|P(h) u\|_{L^2} \mathsf{x} + \tfrac{2}{\inf \alpha} \|b_0\|^2_{L^2} \mathsf{y}\mathsf{z} + \tfrac{2 \|b_1\|^2_{L^\infty}  \| u \|_{L^2} }{\inf \alpha} \mathsf{x}.
\end{split}
\end{equation}
Similar estimates hold for $x \in (-a,0)$ at which $f$ is continuous. Thus $|f(x)|^2 \le C_9 + C_{9} \mathsf{x} + C_{10}\mathsf{y} \mathsf{z}$ at every point of continuity of $f$. But this implies \eqref{bd z square} because
\begin{equation} \label{f square le f square A}
\begin{split}
|f(x)|^2 = |f^A(x)|^2 &= 4^{-1}|f^L(x) + f^R(x)|^2 \\
&\le 2^{-1} (|f^L(x)|^2 + |f^R(x)|^2) = 2^{-1}((|f(x)|^2)^L + (|f(x)|^2)^R) = (|f(x)|^2)^A
\end{split}
\end{equation}
and for all $x \in [-a,a]$, $(|f(x)|^2)^A$ is a limit of values of $|f|^2$ at which $f$ is continuous.

Next, equip $P(h)$ with the domain $\mathcal D_\textrm{max} = \mathcal D$; we show $P(h)$ is symmetric. Let $u, \, v \in \mathcal{D}$. Since $P(h)u = \beta d( -h^2 \alpha u' - i h bu)  + h \beta b D_xu +  u_cV_0$ as distributions and hence as Borel measures, 

\begin{equation} \label{prod to form}
\begin{split}
\langle P(h)u, v \rangle_{L^2(\beta^{-1} dx)} &= \int_{\R} (\overline{P(h)u}) v \beta^{-1}dx \\
&=  \int_{\R} vd( -h^2 \alpha \overline{u}' + ihb\overline{u}) + \int_{\R}[ ih b \overline{u}' v dx   + \overline{u}_c v_c \beta^{-1} V_0] \\
&=  h^2 \int_{\R} \alpha \overline{u}' v' dx + ih \int_{\R} b( \overline{u}' v - \overline{u} v') dx + \int_{\R} \overline{u}_c v_c \beta^{-1} V_0 \\
& \qefed q(u,v).
\end{split}
\end{equation}
Similarly, $\langle u, P(h)v \rangle_{L^2(\beta^{-1} dx)}$ coincides with the third line of \eqref{prod to form}. Thus $P(h)$ is symmetric. 

The last step is to establish that $(P, \mathcal{D})$ is densely defined and $P^* \subseteq P$. For this, define on $H^1(\R)$ the sesquilinear form $q$ as in \eqref{prod to form}. For any $u \in H^1(\R)$, $\|u\|^2_{L^\infty} \le \| u \|_{L^2} \| u'\|_{L^2}$. Therefore, by Young's inequality: 
\begin{equation*}
ab \le p^{-1} \gamma^{1- p} a^p + q^{-1} \gamma b^q \quad \text{for all $\gamma > 0$ and all $p, \, q \ge 1$ such that $p^{-1} + q^{-1} =1$},
\end{equation*}
we have
\begin{equation} \label{bdd below}
\begin{gathered}
 \big| \int |u_c |^2 V_0\big| \le \|V_0\| \|u\|_{\infty}^2 \le \tfrac{\|V_0\|^2}{h^2 \inf \alpha } \| u\|^2_{L^2} + \tfrac{h^2 \inf \alpha }{4 } \| u'\|^2_{L^2} , \\
\begin{split} 2 h \big| \int_{\R} b u \overline{u}'  dx \big| &\le 2h\|b_0\|_{L^2} \|u\|_{L^\infty} \|u'\|_{L^2} + 2h \|b_1\|_{L^\infty} \| u\|_{L^2} \|u'\|_{L^2}  \\
 &\le  \tfrac{864 \| b_0 \|^4_{L^2}}{h^2 \inf^3 \alpha} \| u\|^2_{L^2} + \tfrac{ h^2 \inf \alpha}{8} \| u'\|^2_{L^2}  + \tfrac{8 \| b_1 \|^2_{L^\infty}}{ \inf \alpha} \| u\|^2_{L^2} + \tfrac{ h^2 \inf \alpha}{8}\| u'\|^2_{L^2}. 
 \end{split}
 \end{gathered}
\end{equation}
Note that to estimate the term $2h\|b_0\|_{L^2} \|u\|_{L^\infty} \|u'\|_{L^2} \le 2h \|b_0 \|_{L^2} \| u\|^{1/2}_{L^2} \| u'\|_{L^2}^{3/2}$, we used Young's inequality with $a = \|b_0 \|_{L^2} \| u\|^{1/2}_{L^2}$, $b = \| u'\|_{L^2}^{3/2}$, $p = 4$, $p = 4/3$, and $\gamma = h \inf \alpha /12$. We thus find,
\begin{equation} \label{norm equiv}
q(u,u) \ge - \big( \tfrac{\| V_0\|^2}{h^2 \inf \alpha} +  \tfrac{864 \| b_0 \|^4_{L^2}}{h^2 \inf^3 \alpha} +  \tfrac{4 \|b_1 \|^2_{\infty}}{\inf \alpha} \big) \|u\|^2_{L^2} + \tfrac{h^2 \inf \alpha}{2} \|u'\|^2_{L^2}, 
\end{equation}
so $q$ is semibounded and closed. 
 

By Friedrichs' result \cite[Theorem 2.14]{te}, there is a unique, densely defined, self-adjoint operator $(A, \mathcal{D}_1)$ with 
\begin{equation} \label{D1}
\begin{gathered}
\mathcal{D}_1 = \{ u \in H^1(\R) : \text{there exists $\tilde{u} \in L^2(\R)$ with $q(u,v) = \langle \tilde{u}, v \rangle_{L^2(\beta^{-1}dx)}$ for all $v \in H^1(\R)$} \}, \\
A u  = \tilde{u}.  
\end{gathered}
\end{equation}
By \eqref{prod to form}, this means that, for any $u \in \mathcal{D}_1$,
\begin{equation*}
\int_{\R} \overline{(h^2\alpha u' +ih b u)} v' dx  = \int_{\R} (\overline{A u}  \beta^{-1} -i b \overline{u}') v dx  - \int  \overline{u}_c v_c \beta^{-1} V_0, \qquad v \in H^1(\R).
\end{equation*}
Thus the distributional derivative of $\overline{(-h^2\alpha u' - ihb u)}$ is the measure $(\overline{(Au)} \beta^{-1} - ih b \overline{u}' ) - \overline{u}_c \beta^{-1} V_0$. Proposition \ref{ftc bv prop} then implies $h\alpha u' + ibu$ has locally bounded variation, and $\beta d(-h^2\alpha u' - ihb u) - ih \beta b u' + u_c  V_0 = Au \in L^2(\R)$ Therefore, $(A, \mathcal{D}_1) \subseteq (P, \mathcal{D}_{\max})$, so $P^* \subseteq A^* = A \subseteq P$. Since we already showed $P \subseteq P^*$ (symmetricity), $P^* = P$ as desired. \\
\end{proof} 

\section{Proof of \eqref{pre penult est}} \label{deal with point masses appendix}

We begin from lines two and three of \eqref{just after integrate}. For each $j$, make the change of variable $x \mapsto \eta x + x_j$, yielding
\begin{equation*}
\begin{split}
\sum_{x_j \neq 0}&  \Big[ \int |w_\eta||u|^2 ( \tau \pi^{-1/2} \eta^{-1} W_j  e^{-((x-x_j)/\eta)^2} dx- (1 + \gamma_j^{-1}) \mu_j \delta_{x_j} )\\
&+ \int |w_\eta||h \alpha u' + ibu|^2 ( \pi^{-1/2} \eta^{-1} W_j   e^{-((x-x_j)/\eta)^2} dx- \gamma_j \mu_j \delta_{x_j} ) \Big]  \\
&=   \sum_{x_j \neq 0} \Big[ ( \tau \pi^{-1/2} W_j \int  |w_\eta(x_j+ \eta x)|  |u(x_j + \eta x)|^2 e^{-x^2}dx  - (1 + \gamma_j^{-1}) \mu_j |w_\eta(x_j)| |u(x_j)|^2 ) \\
&+ ( \pi^{-1/2} W_j \int  |w_\eta(x_j+ \eta x)| |(h\alpha u' + ibu)(x_j + \eta x)|^2 e^{-x^2}dx  \\
&- \gamma_j \mu_j |w_\eta(x_j)| |(h\alpha u' + ibu)(x_j )|^2 ) \Big]. \\
\end{split}
\end{equation*}
Thus to find the limit as $\eta \to 0^+$, we must compute  $\lim_{\eta \to 0^+} |w_\eta(x_j)|$ and $\lim_{\eta \to 0^+} |w_\eta(x_j+ \eta x)|$, which by \eqref{weta} rests upon finding $\lim_{\eta \to 0^+} q_{1, \eta} (x_j)$ and $\lim_{\eta \to 0^+} q_{1,\eta} (x_j+ \eta x)$.
From \eqref{q1 BV},
\begin{equation*}
\begin{split}
q_{1, \eta} (x_j) - \sgn(x_j) \int_0^{x_j} \mu_c &= \pi^{-1/2} \eta^{-1} \sgn(x_j) \sum_{x_\ell \neq 0} \int_{0}^{x_j} W_{\ell} e^{-((x'-x_{\ell})/\eta)^2} dx' \\
&= \pi^{-1/2} \sum_{x_\ell \neq 0} \int_{-\sgn(x_j)\frac{x_\ell}{\eta}}^{\sgn(x_j)\frac{x_j-x_\ell}{\eta}}    W_\ell  e^{-(x')^2} dx' \\
 &\to \frac{1}{2} W_j + \sum_{\ell \, : \, \sgn(x_j) x_j > \sgn(x_j) x_{\ell} > 0} W_\ell,
 \end{split}
 \end{equation*}
 and 
 \begin{equation*}
 \begin{split}
q_{1, \eta} (x_j + \eta x) &- \sgn(x_j + \eta x) \int_0^{x_j+ \eta x} \mu_c \\
&= \pi^{-1/2} \eta^{-1} \sgn(x_j + \eta x) \sum_{x_\ell \neq 0} \int^{x_j + \eta x}_0 W_{\ell} e^{-((x'-x_{\ell})/\eta)^2} dx' \\
&=\pi^{-1/2} \sgn(x_j) \sgn(x_j + \eta x) \sum_{x_\ell \neq 0} \int_{-\sgn(x_j) \frac{x_\ell}{\eta}}^{\sgn(x_j)x + \sgn(x_j) \frac{x_j-x_\ell}{\eta}}   W_\ell  e^{-(x')^2} dx'\\
 &\to \pi^{-1/2}  W_j \int^{\sgn(x_j)x}_{-\infty} e^{-(x')^2}dx' + \sum_{\ell \, : \, \sgn(x_j) x_j > \sgn(x_j) x_{\ell} > 0} W_\ell. 
 \end{split}
\end{equation*}
The upshot is
\begin{equation} \label{weta limits}
\begin{gathered}
\lim_{\eta \to 0^+} |w_\eta(x_j)| = e^{\Gamma_j} (e^{q_2(x_j)} -1)e^{W_j/2}, \\
\lim_{\eta \to 0^+} |w_\eta(x_j+ \eta x)| = e^{\Gamma_j} (e^{q_2(x_j)} -1)\exp \big( \pi^{-1/2} W_j \int^{\sgn(x_j)x}_{-\infty} e^{-(x')^2}dx' \big),
\end{gathered}
\end{equation}
where 
\begin{equation*}
\Gamma_j \defeq \sgn(x_j) \int_0^{x_j} \mu_c + \sum_{\ell \, : \, \sgn(x_j) x_j > \sgn(x_j) x_{\ell} > 0} W_\ell. 
\end{equation*}

Now we use \eqref{weta limits} and the dominated convergence theorem to find the limit as $\eta \to 0^+$ of lines two and three of \eqref{just after integrate}. For this step it is helpful to highlight that $u$ is continuous,
\begin{equation*}
\lim_{\eta \to 0^{+}} |(h\alpha u' + ibu)(x_j + \eta x)|^2 = \begin{cases} (|(h\alpha u' + ibu)(x_j)|^2)^L & x < 0, \\
(|(h\alpha u' + ibu)(x_j)|^2)^R & x > 0,
\end{cases}
\end{equation*}
and $\pi^{-1/2} W_j  \exp(\pi^{-1/2}  W_j \int^{x}_{-\infty} e^{-(x')^2}dx') e^{-y^2} = \partial_x( \exp(\pi^{-1/2}  W_j \int^{x}_{-\infty} e^{-(x')^2}dx'))$. We find
\begin{equation} \label{eta limit u terms}
\begin{split}
\tau \pi^{-1/2}& W_j \int  |w_\eta(x_j+ \eta x)| |u(x_j + \eta x)|^2 e^{-x^2}dx  - (1 + \gamma_j^{-1}) \mu_j |w_\eta(x_j)| |u(x_j)|^2 \\
&\to  |u(x_j)|^2 e^{\Gamma_j} ( e^{q_2(x_j)} - 1)  (\tau \pi^{-1/2} W_j \int e^{\pi^{-1/2}  W_j \int^{\sgn(x_j)x}_{-\infty} e^{-(x')^2}dx' } e^{-x^2}dx -(1 + \gamma_j^{-1}) \mu_j e^{\mu_j/2}) \\
&= |u(x_j)|^2 e^{\Gamma_j} ( e^{q_2(x_j)} - 1)(\tau e^{W_j} - 1 - (1 + \gamma_j^{-1}) \mu_j e^{W_j/2}),
\end{split}
\end{equation}
 and 
 
 \begin{equation} \label{eta limit u prime terms 1}
\begin{split}
  \pi^{-1/2}& \int W_j  |w_\eta(x_j+ \eta x)| |(h\alpha u' + ibu)(x_j + \eta x)|^2 e^{-x^2}dx  - \gamma_j \mu_j |w_\eta(x_j)| |(h\alpha u' + ibu)(x_j )|^2  \\
&\to e^{\Gamma_j}( e^{q_2(x_j)} - 1) \big( \pi^{-1/2} W_j (|(h\alpha u' + ibu)(x_j)|^2)^L \int_{-\infty}^0 e^{\pi^{-1/2} \mu_j \int^{\sgn(x_j)x}_{-\infty}  e^{-(x')^2}dx' } e^{-x^2}dx \\ 
&+ \pi^{-1/2}W_j(|(h\alpha u' + ibu)(x_j)|^2)^R \int_{0}^\infty e^{\pi^{-1/2} \mu_j \int^{\sgn(x_j)x}_{-\infty} e^{-(x')^2}dx' } e^{-x^2}dx \\
&-\gamma_j \mu_j |(h\alpha u' + ibu)(x_j)|^2 e^{W_j/2} \big).
\end{split}
\end{equation}
Since 
\begin{equation*} 
\begin{split}
\pi^{-1/2} W_j \int_{-\infty}^0 e^{\pi^{-1/2} W_j \int^{\sgn(x_j)x}_{-\infty}  e^{-(x')^2}dx' } e^{-x^2}dx &= \int_{-\infty}^0 \partial_x ( e^{\pi^{-1/2} W_j \int^{\sgn(x_j)x}_{-\infty}  e^{-(x')^2}dx' } ) dx \\
&= \begin{cases}
e^{W_j/2} - 1 & x_j > 0, \\
e^{W_j} - e^{W_j/2} & x_j < 0,  
\end{cases}
\end{split}
\end{equation*}
and a similar calculation holds for $\pi^{-1/2} W_j \int_{0}^\infty \exp(\pi^{-1/2} \mu_j \int^{\sgn(x_j)x}_{-\infty} e^{-(x')^2}dx' )e^{-x^2}dx$, by \eqref{eta limit u prime terms 1},
 \begin{equation} \label{eta limit u prime terms 2}
\begin{split}
  \pi^{-1/2}& \int W_j  |w_\eta(x_j+ \eta x)| |(h\alpha u' + ibu)(x_j + \eta x)|^2 e^{-x^2}dx  - \gamma_j \mu_j |w_\eta(x_j)| |(h\alpha u' + ibu)(x_j )|^2  \\
&\ge  e^{\Gamma_j} (  e^{q_2(x_j)} - 1) (( |(h\alpha u' + ibu)(x_j)|^2)^A2(e^{W_j/2} - 1) - \gamma_j \mu_j e^{W_j/2} |(h\alpha u' + ibu)(x_j)|^2  ) \\
& \ge e^{\Gamma_j} (  e^{q_2(x_j)} - 1) |(h\alpha u' + ibu)(x_j)|^2  (2(e^{W_j/2} - 1) - \gamma_j \mu_j e^{W_j/2}  ),
\end{split}
\end{equation}
To go from the second to third line of \eqref{eta limit u prime terms 2}, we used $(h\alpha u' + ibu)^A =(h\alpha u' + ibu)$ and that $(|h\alpha u' + ibu|^2)^A \ge |h\alpha u' + ibu|^2$, see \eqref{f square le f square A}.

Inspecting the last line of \eqref{eta limit u terms} and the last line of \eqref{eta limit u prime terms 2}, it is evident that we need to fix the $\gamma_j$ and the $W_j$ so that 
\begin{equation*} 
\begin{gathered}
\tau e^{W_j} -1 - (1 + \gamma_j^{-1}) \mu_j e^{W_j/2} \ge 0, \\
2(e^{W_j/2} -1) - \gamma_j \mu_j e^{W_j/2} \ge 0.
\end{gathered}
\end{equation*}
Take $W_j = M \mu_j$ for $M \ge 1$ sufficiently large to be chosen, and $\gamma_j = e^{-W_j/4}$. So it suffices to have
\begin{equation} \label{fix Wj}
\begin{gathered}
\tau e^{M \mu_j} -1 - 2 \mu_j e^{3M \mu_j /4} \ge 0, \\
2(e^{M \mu_j/2} -1) - \mu_j e^{M \mu_j/4} \ge 0.
\end{gathered}
\end{equation}

 Let us examine the second line of \eqref{fix Wj}:
 
  \begin{equation*}
2(e^{M\mu_j /2} -1) - \mu_j e^{ M \mu_j/4} = e^{M\mu_j/4}(2(e^{M \mu_j /4}- e^{-M \mu_j/4}) - \mu_j ),
\end{equation*}
and 
\begin{equation*}
2(e^{M \mu_j /4}- e^{-M \mu_j/4}) - \mu_j  \ge 2(e^{M \mu_j /4}- 1) - \mu_j \ge (\tfrac{M}{2} - 1) \mu_j 
\end{equation*}
which is nonnegative for $M \ge 2$. Turning to the first line of \eqref{fix Wj}, 
\begin{equation*}
\begin{split}
\tau e^{M \mu_j} -1 - 2 \mu_j e^{3M \mu_j /4}  &= e^{ 3M \mu_j /4}(\tau e^{ M \mu_j /4}  - e^{-3 M \mu_j /4} -  2 \mu_j ) \\
&\ge e^{ 3 M \mu_j /4} (\tfrac{M \tau }{4} - 2 ) \mu_j
\end{split}
\end{equation*}
is nonnegative for $M \ge 8 \tau^{-1}$. Therefore, taking $M = \max(2, 8 \tau^{-1})$ yields \eqref{pre penult est}.

\section{Simple operators with no resonance at zero} \label{no zero resonance appendix}

In this appendix we give simple examples of operators $H$ as in \eqref{H} that do not have a resonance at zero.
Take $\alpha = \beta = 1$, and let $V = M \mathbf{1}_{[-1,1]}$ , $b = \mathbf{1}_{[-1,1]}$ be indicator functions, where $M > 0$ is to be chosen. Let $u$ in the domain of $H$ have the form $u \in (H - i\ep)^{-1} (|x| + 1)^{-\frac{3 +\delta}{2}} L^2(\R)$ for $\ep, \delta > 0$. Then by \eqref{prod to form}, for all $\gamma > 0$,
\begin{equation} \label{choose M here}
\begin{split}
\frac{1}{2\gamma } \|  (|x| + 1)^{\frac{3 +\delta}{2}}(H - i\ep)u\|^2_{L^2} &+ \frac{\gamma }{2} \| (|x| + 1)^{-\frac{3 +\delta}{2}}u\|^2_{L^2}  \\
&\ge \real \langle (H - i \ep)u, u \rangle_{L^2} \\
 &= \int_{\R}  |u'|^2 dx -2 \imag \int_{-1}^1 \overline{u}'u  dx + M \int_{-1}^1 |u|^2 dx \\
&\ge \frac{1}{2} \| u'\|_{L^2}^2+ (M - 2) \|u\|^2_{L^2[-1,1]}.
\end{split}
\end{equation}
On the other hand
\begin{equation*}
\begin{split}
\int& (|x| + 1)^{-3-\delta} |u|^2 dx \\
&= \frac{1}{2 + \delta} \big( \int^0_{-\infty} \partial_x ((-x + 1)^{-2 -\delta}) |u|^2 dx -\int^\infty_0 \partial_x ((|x| + 1)^{-2 -\delta}) |u|^2 dx  \big) \\
&= \frac{2}{2 + \delta} |u(0)|^2 + \frac{2}{2 + \delta} \real \big(  \int_0^\infty(x + 1)^{-2 -\delta} \overline{u}'u dx - \int_{-\infty}^0(-x + 1)^{-2 -\delta} \overline{u}'u dx \big) \\
& \le \frac{2}{2 + \delta} |u(0)|^2 + \frac{1}{2} \int (|x| + 1)^{-3-\delta} |u|^2 dx + \frac{1}{2 + \delta} \int (|x| + 1)^{-1-\delta} |u'|^2 dx.
\end{split}
\end{equation*}
Now use that, for any $v \in H^1[-1,1]$, $\| v\|^2_{L^\infty[-1,1]} \le 2^{-1} \|v \|^2_{L^2[-1,1]} + 2 \|v \|_{L^2[-1,1]} \|v' \|_{L^2[-1,1]}$ \cite[Problem 2.33]{te},
\begin{equation} \label{IBP to control wtd norm}
\begin{split}
\int &(|x| + 1)^{-3-\delta} |u|^2 dx \\
&\le \frac{2}{2 + \delta} \big( \|u \|^2_{L^2[-1,1]} +  4\| u \|_{L^2[-1,1]} \| u' \|_{L^2[-1,1]} \big) +  \frac{2}{2 + \delta} \int (|x| + 1)^{-1-\delta} |u'|^2 dx \\
&\le \frac{2}{2 + \delta} \big( 3\|u \|^2_{L^2[-1,1]} +   2\| u' \|^2_{L^2[-1,1]} \big) +  \frac{2}{2 + \delta} \int (|x| + 1)^{-1-\delta} |u'|^2 dx.
\end{split}
\end{equation}
Hence
\begin{equation} \label{control wtd norm}
 \frac{2+ \delta}{12}\int (|x| + 1)^{-3-\delta} |u|^2 dx \le \frac{1}{2}\|u \|^2_{L^2[-1,1]} +    \frac{1}{2}\| u' \|^2_{L^2}.
\end{equation}
Thus, if we $M$ choose large enough in \eqref{choose M here}, in combination with \eqref{control wtd norm} we get
\begin{equation*}
 \frac{2+ \delta}{12}\int (|x| + 1)^{-3-\delta} |u|^2 dx \le \frac{1}{2\gamma } \|  (|x| + 1)^{\frac{3 +\delta}{2}}(H - i\ep)u\|^2_{L^2} + \frac{\gamma }{2} \| (|x| + 1)^{-\frac{3 +\delta}{2}}u\|^2_{L^2}.
 \end{equation*}
Selecting $\gamma$ small enough yields, for $C > 0$ independent of $\ep$ and $u$,
\begin{equation*}
\| (|x| + 1)^{-\frac{3+\delta}{2}} u \|_{L^2} \le C \|  (|x| + 1)^{\frac{3 +\delta}{2}}(H - i\ep)u\|^2_{L^2}.
 \end{equation*}
 This estimate implies the cutoff resolvent does not have a zero resonance. A similar estimate can be performed if $b = 0$ and $V = M \delta_0$, i.e., $V$ is the dirac measure of mass $M$ concentrated at zero. In that case the last term of line three of \eqref{choose M here} becomes $M |u(0)|^2$, and can be used to control the boundary term that appears after integrating by parts in \eqref{IBP to control wtd norm}.

\end{document}